\documentclass[journal]{IEEEtran}
\usepackage{amssymb,amsmath,amsthm, amsfonts}
\usepackage{subfigure}
\usepackage{graphicx}
\usepackage{graphics}
\usepackage{subfigure}
\usepackage{epsfig}
\usepackage{color}
\usepackage{multicol}
 \usepackage{mathrsfs}
 \allowdisplaybreaks[4]
 \usepackage{dsfont}
\usepackage{epstopdf}
\usepackage{caption}
\def\dref#1{(\ref{#1})}
\def\disp{\displaystyle}

\theoremstyle{plain}
\theoremstyle{definition}
 \newtheorem{definition}{Definition}

\newtheorem{remark}{Remark}

\newtheorem{theorem}{Theorem}
\newtheorem{lemma}{Lemma}

\newtheorem{assumption}{Assumption}

\numberwithin{equation}{section}

\begin{document}

\title{Output Consensus of Heterogeneous Multi-Agent Systems with Mismatched Uncertainties and Measurement Noises: An ADRC Approach}

\author{Mengling Li, Ze-Hao Wu, Feiqi Deng, and  Zhi-Liang Zhao
\thanks{Mengling Li and  Ze-Hao Wu are with School of Mathematics and Big Data, Foshan University, Foshan 528000, China.
        {\tt\small  Emails: menglingli@fosu.edu.cn; zehaowu@amss.ac.cn} }
 \thanks{Feiqi Deng  is with Systems Engineering Institute, South China University of Technology, Guangzhou 510640, China.
 {\tt\small  Email: aufqdeng@scut.edu.cn}}
 \thanks{Zhi-Liang Zhao is with School of Mathematics and Information Science, Shanxi Normal University, Xi'an 710119, China.
 {\tt\small  Email: zhiliangzhao@snnu.edu.cn} }

 }

%

%
%

\markboth{}%
{Shell \MakeLowercase{\textit{et al.}}: Bare Demo of IEEEtran.cls
for Journals}
%



\maketitle

\begin{abstract}
In this paper, the practical output consensus problem for heterogeneous high-order leader-follower  multi-agent systems under directed communication topology
containing a directed spanning tree and subject to large-scale mismatched disturbances, mismatched uncertainties, and measurement noises is addressed.
By introducing a reversible state transformation without changing the output,
the actual total disturbance affecting output performance of each agent and matched
with the control input of the transformed system is extracted and estimated by extended state observers.
Then, the control protocols based on estimates of extended state observers, are designed by combing the output feedback control ones to obtain output
consensus and feedforward compensators to attenuating the total disturbance of each agent actively.
It is shown with a rigorous proof that the outputs of all followers can track practically the output of the leader,
and all the states of the leader-follower multi-agent systems are bounded. Some numerical simulations
are performed to verify the validity of the control protocols and theoretical result.
\end{abstract}


%
\IEEEpeerreviewmaketitle

\section{Introduction}
%
%
%
%

\IEEEPARstart{O}ver the last few years, cooperative control for multi-agent systems (MASs) has been getting great interests owing to its wide  prospect for
applications in sensor networks, cooperation of
multi-robot teams, coordination of unmanned aerial vehicles and so on \cite{Cao2015}.
Consensus, meaning that states or outputs of agents converge to the same value, is known as a fundamental problem in the field of cooperative control.
The consensus control of MASs has received considerable number of concerns in the control community, see for instance \cite{yu2010,yu2011,fan2014,limulti-agent,2nonlinearmulti-agent,3stochasticmulti-agent,zong2019,dengcwuzg2020}.
 An important consensus control strategy is the leader-follower coordination control among a set of agents which has been widely used in many applications such as  unmanned aerial vehicle formation \cite{In6}, communication systems \cite{In7}, vehicular networks \cite{In8}, power engineering \cite{powerapplication}, to name just a few.
 In additions, agents are usually under complex working environment subject to disturbances and uncertainties in practical applications.
Therefore, as yet, various anti-disturbance control approaches have been put to use in the leader-follower consensus of uncertain MASs,
 such as adaptive control \cite{liTAC2014,adptiveconl2017}, sliding mode control \cite{In13,slidecontrol2},  fuzzy adaptive dynamic programming \cite{zhangh2014},
 and the distributed internal model principle for output regulation \cite{outputregulation}, etc.

Nevertheless,  most of aforementioned proposed consensus control methods are the passive anti-disturbance control
 ones, achieving the disturbance rejection objective only by feedback control based on tracking errors.
 These consensus protocols are not fast and direct when coping with large scale disturbances and uncertainties,
 compared with other two representative active anti-disturbance control methods well-known as
the active disturbance rejection control (ADRC) \cite{ADRC} and disturbance observer-based control (DOBC) \cite{DOBC} with
extensive engineering applications.
ADRC, as a novel active anti-disturbance control technology, was initiated by Han \cite{ADRC}.
 The central constituent of  ADRC is the extended state observer (ESO),
 aiming at estimation of both unmeasured states and the total disturbance representing
 are total effects of all disturbances and uncertainties affecting system performance.
 Based on the estimate of the total disturbance,
 the ADRC controller, a compound one, is comprised of a feedback controller and
  a feedforward compensator via ESO, where the compensator takes great effect in the disturbance
  rejection, and the feedback controller can be designed individually to obtain the control objective of nominal systems.
On the whole, because of this estimation/cancellation
characteristic in the ADRC framework, the total disturbance can be actively and quickly rejected without ruining the performance of nominal systems,
and the ADRC controller is not conservative.

Recently the theoretical foundation of the active anti-disturbance control to the consensus problem for MASs with matched disturbances
and uncertainties have been well developed, see for instance \cite{Cao2015,ESOmulti-agentsystems,zou2022,limengling2022}.
However, in some practical processes, the performance of each agent may be affected by  the mismatched disturbances and uncertainties different from the control input channels \cite{hua2018}, such as magnetic levitation vehicle systems \cite{mismatched1} and missile systems \cite{mismatched2}. The recent progresses concerning the ADRC approach to
stabilization and output tracking of uncertain nonlinear systems with mismatched disturbances and uncertainties can
be founded in \cite{xue2014,guo2017,zhao2017,chens2020,wu2021} and the references therein, and it has been developed to the MASs counterpart thereafter.
For example, without considering system uncertainties and under the assumption that
states are measurable, the output consensus control problem for homogeneous
higher-order leader-follower MASs with mismatched disturbances has been investigated by designing
 feedforward-feedback composite consensus controls based on the sliding-mode control (SMC) and DOBC methods \cite{misMAS1},
and based on the backstepping strategy and a generalized proportional-integral observer \cite{misMAS2};
Under the assumption that states are measurable, the formation tracking
problem for nonaffine nonlinear homogeneous MASs with communication delays and system uncertainties has been addressed in \cite{Yandyue};
Under a connected undirected network, the consensus problem has been investigated
 via the ADRC method for nonlinear heterogeneous MASs  subject to
bandwidth limitation, mismatched disturbances and uncertainties \cite{ran2021}.
However, to the authors' knowledge, for nonlinear homogeneous MASs under directed communication topology and
with large-scale mismatched disturbances and uncertainties, a comprehensive output consensus protocol design via the active anti-disturbance control approaches
and the theoretical analysis are still not resolved, and few relevant literature addresses measurement noises.

Motivated by the current research status, in this paper we apply the
ADRC approach  to solve  the practical output consensus
problem for a kind of nonlinear heterogeneous  high-order leader-follower  MASs with mismatched disturbances,
mismatched uncertainties, and measurement noises. The contribution and novelty
are twofold as follows: a) The nonlinear heterogeneous high-order leader-follower MASs under directed communication topology are
subject to disturbances and uncertainties in large scale including unmeasurable states,
mismatched disturbances, mismatched uncertainties, and measurement noises, only with the output
of each agent be available for ADRC designs; b) the ADRC consensus protocols are designed to
obtain disturbance rejection in an active way and practical output consensus of the uncertain MASs,
with a rigorous theoretical foundation be presented.

The structure of this paper will be proceeded as below. The problem formulation and preliminaries are
given in Section \ref{Se2}.
The ADRC consensus protocols design and the main practical output
consensus result with its proof be presented in Section \ref{Se3}. In Section \ref{se4}, several
simulations are demonstrated to authenticate the rationality of the
ADRC consensus protocols and the theoretical result, and ultimately the concluding remarks
are stated in Section \ref{se5}.

\section{Problem formulation and preliminaries}\label{Se2}

Throughout the paper, some mathematical notations are agreed as follows.
$\lambda_{\min}(Z)$ and  $\lambda_{\max}(Z)$  represent the minimum and maximum eigenvalues of a
 matrix $Z$, respectively; $|\cdot|$ represents the absolute value of scalars, and
$\|\cdot\|$ denotes the $2$-norm of matrices or vectors;
${ 1}_{m}$ and ${ 0}_{m}$ denote, respectively, the $m\times 1$
column vector with all elements be ones and zeros; ${ 0}_{m\times n}$ denotes the zero matrix with  $m$ rows and $n$ columns, and $\mathbb{I}_{m}$ denotes the $m$ identity
matrix; $\mathrm{diag}(\mathfrak{p}_{1},\cdots,\mathfrak{p}_{m})$ implies the diagonal matrix with
diagonal entries be $\mathfrak{p}_{1},\cdots,\mathfrak{p}_{m}$;
$\otimes$ denotes the Kronecker product satisfying the following several properties for some matrices $\mathbb{A}_{i}\;(i=1,2,3,4)$ with appropriate dimensions:
\begin{eqnarray*}
&&\hspace{-0.5cm}(\mathbb{A}_{1}\otimes \mathbb{A}_{2})^\top=\mathbb{A}^\top_{1}\otimes \mathbb{A}^\top_{2},\;\; (\mathbb{A}_{1}\otimes \mathbb{A}_{2})^{-1}=\mathbb{A}^{-1}_{1}\otimes \mathbb{A}^{-1}_{2}, \cr&&\hspace{-0.5cm} (\mathbb{A}_{1}+\mathbb{A}_{2})\otimes \mathbb{A}_{3}=\mathbb{A}_{1}\otimes \mathbb{A}_{3}+\mathbb{A}_{2}\otimes \mathbb{A}_{3}, \cr&&\hspace{-0.5cm} \mathbb{A}_{1}\otimes (\mathbb{A}_{2}+\mathbb{A}_{3})=\mathbb{A}_{1}\otimes \mathbb{A}_{2}+\mathbb{A}_{1}\otimes \mathbb{A}_{3},\cr&& \hspace{-0.5cm}(\mathbb{A}_{1}\otimes \mathbb{A}_{2})(\mathbb{A}_{3}\otimes \mathbb{A}_{4})=\mathbb{A}_{1}\mathbb{A}_{3}\otimes \mathbb{A}_{2}\mathbb{A}_{4}.
\end{eqnarray*}
These conventional properties will be used frequently in the following proof.

Next,  some mathematical definitions and simple explanations for topology graph are given. Consider a MAS with $m(\geq 1)$ followers agent(s) and one leader agent. $\mathcal{N}=\{1,\cdots,m\}$ and $\mathcal{M}=\{0\}$ stand for the set of followers and leader, respectively. $\bar{\mathcal{N}}=\mathcal{N}\cup\mathcal{M}$. The network topology among the followers and leader  is symbolised by a directed graph $\mathcal{G}=\{\mathds{V}, \mathds{E}\}$, where $\mathds{V}=\{\mathds{V}_0,\mathds{V}_1,\cdots,\mathds{V}_{m}\}$ indicates the set of vertices
 denoting the above $m+1$ agents and $\mathds{E}\subseteq\mathds{V}\times\mathds{V}$ designates the set of edges of the graph. The directed edge $\mathds{E}_{ij}=(\mathds{V}_i,\mathds{V}_j)$ indicates that the vertex $\mathds{V}_j$ can
 receive information from vertex $\mathds{V}_i$.  Denote $\mathcal{A}=[a_{ij}]\in \mathbb{R}^{(m+1)\times (m+1)}$ by the weighted adjacency matrix of $\mathcal{G}$, where $a_{ij}=1$ is equivalent to $\mathds{E}_{ji}\in\mathds{E}$, otherwise $a_{ij}=0$. And for any $i\in \bar{\mathcal{N}}$, $a_{ii}=0$. Let $\mathcal{N}_i=\{\mathds{V}_j\in \mathds{V}|\mathds{E}_{ji}\in\mathds{E}\}$ be the set of in-neighbors of vertex $\mathds{V}_i$ and $D=\mathrm{diag}\{D_0,\cdots, D_{m}\}\in \mathbb{R}^{(m+1)\times (m+1)}$ represent the in-degree matrix with $D_i=\sum_{j\in \mathcal{N}_i}a_{ij}$ being the weighted in-degree of agent $i$.
The Laplacian matrix is defined as $L=[l_{ij}]=D-\mathcal{A}$ that can be represented as $L=\left[\begin{matrix}
0 & 0_{1\times m} \\
L_0 &L_1
\end{matrix}\right]$ with $L_0\in \mathbb{R}^{m\times 1}$ and $L_1\in \mathbb{R}^{m\times m}$
because the leader has no in-neighbors. $a_{i0}>0$ indicates that the $i$-th follower agent can obtain the information of the leader, otherwise $a_{i0}=0$.

For the topology in this paper, we give the following assumption.

\begin{assumption}\label{a1}
The topology $\mathcal{G}$ contains a directed spanning tree and the leader is the root.
\end{assumption}

\begin{lemma}\cite{liTAC2014}\label{le1}
Under Assumption \ref{a1}, $L_1$ is a nonsingular diagonally dominant $M$-matrix, so there is a positive definite diagonal matrix $W=\mathrm{diag}\{W_1,\cdots,W_m\}$, where $(W_1,\cdots,W_m)^\top=(L_1^\top)^{-1}1_m$, such that $WL_1+L_1^\top W$ is positive definite.
\end{lemma}

The above lemma is a very useful lemma for matrix $L_1$ and will facilitate our analysis. In this paper, we consider the heterogeneous high-order MASs containing $m$ agents, and the
dynamics of the $i$-th agent $(i\in \{1,\cdots,m\})$ subject to
mismatched disturbances, mismatched uncertainties, and measurement noises are described as
\begin{equation}\label{s1}
\hspace{-0.2cm}\left\{\begin{array}{l}
\dot{x}_{i1}(t)=x_{i2}(t)+h_{i1}(x_{i1}(t),d_{i}(t)),\cr
\hspace{1.2cm}  \vdots \cr
\dot{x}_{i,n-1}(t)=x_{in}(t)+h_{i,n-1}(x_{i1}(t),\cdots,x_{i,n-1}(t),d_{i}(t)),\\
\dot{x}_{in}(t)=h_{i,n}(x_{i1}(t),\cdots,x_{in}(t),d_{i}(t))+u_i(t),\\
y_{i}(t)=x_{i1}(t)+w_{i}(t), \;\; i\in \{1,\cdots,m\},
\end{array}\right.
\end{equation}
where $x_i(t)=(x_{i1}(t),\cdots, x_{in}(t))^\top\in \mathbb{R}^n$, $d_{i}(t)\in \mathbb{R}$, and $w_{i}(t)\in \mathbb{R}$ are the system state, external disturbance, and measurement noise
of the $i$-th agent, respectively;
 $h_{ij}:\mathbb{R}^{j+1}\rightarrow \mathbb{R}\;(j=1,\cdots,n-1)$ and $h_{in}:\mathbb{R}^{n+1}\rightarrow \mathbb{R}$ are unknown system functions;
$u_i(t)\in \mathbb{R}$ and $y_i(t)\in \mathbb{R}$ are control input and output measurement of the $i$-th agent, respectively;
 $x(t)\triangleq(x_1^\top(t),\cdots,x_m^\top(t))^\top$ is the state of the MASs.
The dynamics of the leader is described as
\begin{equation}\label{s4}
\begin{cases}
\dot{x}_{0k}(t)=x_{0,k+1}(t),\; k=1,\cdots,n-1,\\
\dot{x}_{0,n}(t)=u_0(t),\\
y_0(t)=x_{01}(t),
\end{cases}
\end{equation}
where $u_0(t)$ is the control input of the above leader system, and   $x_{0}(t)=(x_{01}(t),\cdots,x_{0n}(t))^{\top}$.

The aim of this paper is to design consensus protocols based on the ADRC approach to enable that the outputs of all the followers can track practically the output of the leader, and all  states of the leader-follower MASs are bounded.

\section{ADRC Consensus Protocols design and the main result}\label{Se3}
 
The actual total disturbance of each agent that will be estimated by ESO contains high order derivatives of the mismatched disturbances, mismatched uncertainties, and measurement
noises. To facilitate the following ESOs designs and theoretical analysis, the system functions $h_{ij}$'s, external disturbances $d_{i}(t)$'s, and measurement noises $w_{i}(t)$'s
are required to satisfy some smooth assumption as follows.

\begin{assumption}\label{a22}
$h_{ij}\in C^{n+1-j}(\mathbb{R}^{j+1};\mathbb{R})$, and $d_{i}(t)$'s and $w_{i}(t)$'s are  $n$-th continuously differentiable and
 $(n+1)$-th continuously differentiable with regard to the  $t$, respectively.
\end{assumption}

By Assumption \ref{a22}, we can introduce the following state transformation by setting
\begin{equation}\label{s2}
\begin{cases}
\bar{x}_{i1}(t)= x_{i1}(t)+w_{i}(t),\\
\bar{x}_{ij}(t)=x_{ij}(t)+\disp \sum_{l=1}^{j-1}h_{i,j-l}^{(l-1)}(x_{i1}(t),\cdots,x_{i,j-l}(t),d_{i}(t))\\
\qquad\qquad +w_{i}^{(j-1)}(t),\;\;
j=2,\cdots,n,\; i=1,\cdots,m,
\end{cases}
\end{equation}
where $h_{i,j-l}^{(l-1)}(x_{i1}(t),\cdots,x_{i,j-l}(t),d_{i}(t))$ represent the $(l-1)$-th
 derivatives of $h_{i,j-l}(x_{i1}(t),\cdots,x_{i,j-l}(t),d_{i}(t))$ with regard to the time variable $t$
 and $h_{i,j-1}^{(0)}(x_{i1}(t),\cdots,x_{i,j-1}(t),d_{i}(t))\triangleq h_{i,j-1}(x_{i1}(t),\cdots,x_{i,j-1}(t),d_{i}(t))$, similarly hereinafter,
 and we set $\bar{x}_{i}(t)=(\bar{x}_{i1}(t),\cdots,\bar{x}_{in}(t))^{\top}$ and $\bar{x}(t)=(\bar{x}_{1}(t),\cdots,\bar{x}_{m}(t))^{\top}$ in what follows.
 It follows easily from Assumption \ref{a22} and the lower triangular structure of MASs \dref{s1} that there are continuous functions $\phi_{ij}$ such that

 \begin{eqnarray}\label{computations}
 &&\hspace{-1.2cm} h_{i,j-l}^{(l-1)}(x_{i1}(t),\cdots,x_{i,j-l}(t),d_{i}(t))\cr&&\hspace{-1.2cm}=\phi_{ij}(x_{i1}(t),\cdots,x_{i,j-1}(t),d_{i}(t),\cdots,d^{(l-1)}_{i}(t)).
 \end{eqnarray}
Thus, it can be further obtained that there are continuous functions $\psi_{ij}\;(i=1,\cdots,m,j=1,\cdots,n)$, such that
\begin{equation}\label{s22}
\begin{cases}
x_{i1}(t)=\bar{x}_{i1}(t)-w_{i}(t)\triangleq \psi_{i1}(\bar{x}_{i1}(t),w_{i}(t)),\\
x_{i2}(t)=\bar{x}_{i2}(t)-h_{i1}(x_{i1}(t),d_{i}(t))-\dot{w}_{i}(t)\\\qquad
\triangleq \psi_{i2}(\bar{x}_{i1}(t),\bar{x}_{i2}(t),d_{i}(t),\dot{w}_{i}(t)),\cr
\hspace{1.2cm}  \vdots \cr
x_{in}(t)=\bar{x}_{in}(t)-\sum_{l=1}^{n-1}h_{i,n-l}^{(l-1)}(x_{i1}(t),\cdots,\\
\qquad\qquad x_{i,n-l}(t),d_{i}(t))-w_{i}^{(n-1)}(t)\\
\qquad\triangleq \psi_{in}(\bar{x}_{i1}(t),\cdots,\bar{x}_{in}(t),d_{i}(t),\cdots,d_{i}^{(n-2)}(t),\\
\qquad\qquad\quad w_{i}(t),\cdots,w_{i}^{(n-1)}(t)),\;
i=1,\cdots,m,
\end{cases}
\end{equation}
which can be equivalently expressed as
\begin{eqnarray}\label{26equnation}
&&\hspace{-1cm}x_{i}(t)=\psi_{i}(\bar{x}_{i1}(t),\cdots,\bar{x}_{in}(t),d_{i}(t),\cdots,d_{i}^{(n-2)}(t),\cr&&\hspace{0.2cm} w_{i}(t),\cdots,w_{i}^{(n-1)}(t)),\;i=1,\cdots,m,
\end{eqnarray}
where
\begin{eqnarray}
\psi_{i}\triangleq (\psi_{i1},\psi_{i2},\cdots,\psi_{in})^{\top}.
\end{eqnarray}

 With $(\bar{x}_{i1}(t),\cdots,\bar{x}_{in}(t))$ being the new state variables,  MASs (\ref{s1})
 subject to mismatched disturbances, mismatched uncertainties, and measurement noises is transformed to be
\begin{equation}\label{s3}
\begin{cases}
\dot{\bar{x}}_{i1}(t)=\bar{x}_{i2}(t),\cr
\hspace{1.2cm}  \vdots \cr
\dot{\bar{x}}_{i,n-1}(t)=\bar{x}_{in}(t),\\
\dot{\bar{x}}_{in}(t)=\bar{x}_{i,n+1}(t)+u_i(t),\\
y_i(t)=\bar{x}_{i1}(t),\;\; i\in \{1,\cdots,m\},
\end{cases}
\end{equation}
which is subject to an actual total disturbance (extended state) matched with
the control input given by
\begin{eqnarray}\label{totaldisturbance}
&&\hspace{-1cm}\bar{x}_{i,n+1}(t)\triangleq h_{in}(x_{i1}(t),\cdots,x_{in}(t),d_{i}(t))\cr&&\hspace{-1cm}+\sum_{l=1}^{n-1}h_{i,n-l}^{(l)}(x_{i1}(t),\cdots,x_{i,n-l}(t),d_{i}(t))+w_{i}^{(n)}(t).
\end{eqnarray}
Similar to \dref{computations}, it follows from Assumption \ref{a22} and the lower triangular structure of MASs \dref{s1} that there exist continuous function $\varphi_{i}$ such that
\begin{eqnarray}\label{exprssiondisutranve}
&&\bar{x}_{i,n+1}(t)=\varphi_{i}(x_{i1}(t),\cdots,x_{in}(t),d_{i}(t),\cdots,\cr&&\hspace{1.8cm}d^{(n-1)}_{i}(t))+w_{i}^{(n)}(t).
\end{eqnarray}

To be emphasized, aforementioned transformation keeps the same output $y_{i}(t)$
between   MASs \dref{s1}  and MASs (\ref{s3}), and $\bar{x}_{i,n+1}(t)$ is the actual total disturbance affecting the output $y_{i}(t)$
of $i$-th agent, which can be observed from the output $y_{i}(t)$. Actually, the uncertain MASs (\ref{s3}) is said to exactly observable
if $y_{i}(t)\equiv 0, u_{i}(t)\equiv 0, \;\forall t\in [0,\infty)\Rightarrow \bar{x}_{i,n+1}(t)\equiv 0,\;\forall t\in [0,\infty)$
and $\bar{x}_{ij}(0)=0,\;j=1,\cdots,n$, see, e.g., \cite[p.5, Definition 1.2]{exactobserval}. Therefore, next we design a set of ESOs
to estimate  the total disturbance $\bar{x}_{i,n+1}(t)$ and states of MASs  (\ref{s3}) in real time as follows
\begin{equation}\label{s6}
\begin{cases}
\dot{\hat{\bar{x}}}_{i1}(t)=\hat{\bar{x}}_{i2}(t)+k_1r(y_i(t)-\hat{\bar{x}}_{i1}(t)),\cr
\hspace{1.2cm}  \vdots \cr
\dot{\hat{\bar{x}}}_{i,n-1}(t)=\hat{\bar{x}}_{in}(t)+k_{n-1}r^{n-1}(y_i(t)-\hat{\bar{x}}_{i1}(t)),\\
\dot{\hat{\bar{x}}}_{in}(t)=\hat{\bar{x}}_{i,n+1}(t)+k_nr^n(y_i(t)-\hat{\bar{x}}_{i1}(t))+u_i(t),\\
\dot{\hat{\bar{x}}}_{i,n+1}(t)=k_{n+1}r^{n+1}(y_i(t)-\hat{\bar{x}}_{i1}(t)),\;
 i=1,\cdots,m,
\end{cases}
\end{equation}
where $r$ is the tuning gain, the constants $k_i$'s are selected such that
\begin{equation}\label{umatrix}
U\triangleq\left[\begin{matrix}
-k_1 &1 &0&\cdots&0\\
-k_2&0&1&\cdots&0\\
\vdots&\vdots&\vdots&\vdots&\vdots\\
-k_{n+1}&0&0&\cdots&0
\end{matrix}\right]
\end{equation}
is Hurwitz, and $\hat{\bar{x}}_{ij}(t)\;(i=1,\cdots,m,j=1,\cdots,n)$ are the estimates of $\bar{x}_{ij}(t)$.
From here and the whole paper, we always drop $r$ for the solutions of \dref{s6} and other systems
by abuse of notation without confusion.

Set
\begin{equation*}
A=\left[\begin{matrix}
0^{(n-1)\times1}& \mathbb{I}_{n-1}\\
0&0^{1\times (n-1)}
\end{matrix}\right]\in \mathbb{R}^{n\times n}, B=\left[\begin{matrix}
0\\
\vdots\\ 0 \\
1
\end{matrix}\right]\in \mathbb{R}^{n\times 1}.
\end{equation*}

It can be easily checked that $(A,B)$ is  stabilizable, and then
 there is a positive definite matrix $P\in \mathbb{R}^{n\times n}$ to the following Riccati equation:
\begin{equation}\label{p1}
A^\top P+PA-\mu_0 PBB^\top P=-\mathbb{I}_{n},
\end{equation}
where $\mu_0\triangleq\mu\lambda_{\min}(W^{-1})$ with $\mu\triangleq\lambda_{\min}(WL_1+L_1^TW)$ and
 $W$ be specified in Lemma \ref{le1}.

 Set
\begin{equation}\label{ens}
\vartheta_{iq}(t)=\sum_{j\in {\mathcal{N}_{i}}}a_{ij}(\hat{\bar{x}}_{iq}(t)-\hat{\bar{x}}_{jq}(t))+a_{i0}(\hat{\bar{x}}_{iq}(t)-x_{0q}(t)),
\end{equation}
where $q=1,\cdots,n$, and we set $\vartheta_i=(\vartheta_{i1},\cdots,\vartheta_{in})^\top$.

Then, the ADRC consensus protocols can be designed as
\begin{equation}\label{con}
u_i(t)=\mathrm{sat}_{M}(K\vartheta_i(t))-\mathrm{sat}_{N_{i}}(\hat{\bar{x}}_{i,n+1}(t))+u_0 (t),
\end{equation}
for $i=1,\cdots,m$, where $K=-B^{\top}P\in \mathbb{R}^{1\times n}$ is the output feedback control gain vector with $P$ be specified in (\ref{p1}), $M$ and
$N_{i}$'s are positive constants specified respectively in  \dref{Mconstant} and  \dref{Nconstant}, and the
continuous differentiable saturation odd function $\mathrm{sat}_{\Theta}: \mathbb{R}\rightarrow \mathbb{R}$ for any given $\Theta>0$ is defined by
\begin{equation}\label{saturaionfunct}
\mathrm{sat}_{\Theta}(s)=\begin{cases}
s, \qquad\qquad \qquad\qquad\qquad\quad0\leq s\leq {\Theta},\\
-\frac{1}{2}s^2+({\Theta}+1)s-\frac{1}{2}{\Theta}^2,\;\;\;{\Theta}<s\leq {\Theta}+1,\\
{\Theta}+\frac{1}{2}, \qquad\qquad \qquad\qquad\quad s>{\Theta}+1.
\end{cases}
\end{equation}
The ADRC consensus protocols \dref{con} are composed of a common output feedback control $\mathrm{sat}_{M}(K\vartheta_i(t))$
to achieve the output consensus of MASs, a compensator $-\mathrm{sat}_{N_{i}}(\hat{\bar{x}}_{i,n+1}(t))$
based on the ESOs \dref{s6} to compensate for the actual total disturbance, and a feedforward signal $u_{0}(t)$ obtained from
the leader. Compared with conventional passive anti-disturbance consensus protocols using only feedback control designs,
the ADRC consensus protocols \dref{con} play an important role in disturbance rejection actively by the compensator $-\mathrm{sat}_{N_{i}}(\hat{\bar{x}}_{i,n+1}(t))$
based on ESOs \dref{s6}.
There are two main reasons for using the saturation functions $\mathrm{sat}_{M}(\cdot)$ and $\mathrm{sat}_{N_{i}}(\cdot)$ in the ADRC consensus protocols \dref{con}.
On the one hand, it can avoid the peaking value phenomenon near the initial state brought about by the possible high gain in ESO \dref{s6}.
On the other hand, since the ADRC consensus protocols \dref{con} are always bounded because of the saturation functions,
it will be advantageous to the following theoretical analysis.

The practical output consensus to be obtained is defined as follows.
\begin{definition}\label{defdefint1}
The practical output consensus of the leader-follower MASs (\ref{s1})-(\ref{s4}) is said to be solved
if the consensus protocols (\ref{con}) dependent on the
tuning observer gain $r$ are available such that for any $\varepsilon>0$ and any
initial values in a given compact set, there exists $r^*\geq 1$, such that for any $r\geq r^{*}$, there holds
\begin{equation}\label{outputdefintion}
|y_i(t)-y_{0}(t)|\leq \varepsilon,\;\forall t\geq t_{r}, i=1,\cdots,m,
\end{equation}
where $t_{r}$ is a positive constant dependent on $r$.
\end{definition}

\begin{remark}
Definition \ref{defdefint1} means directly that
for any $\varepsilon>0$, $\disp\lim_{t\rightarrow+\infty}|y_i(t)-y_{0}(t)|\leq\varepsilon$
when $r$ is tuned to be large accordingly. More specifically,
the practical output consensus indicates that the errors between outputs of followers
and output of the leader can be ensured to be arbitrarily close to zero at the steady state,
provided that corresponding $r$-dependent ESOs-based consensus protocols (\ref{con}) are designed.
\end{remark}

To achieve practical output consensus and boundedness of the leader-follower MASs (\ref{s1})-(\ref{s4}), the following assumptions
are required additionally.

\begin{assumption}\label{as3}
There are a few positive constants $\alpha_1,\alpha_2,\alpha_{3i}$ and bounded control input $u_0(t)$ such that
$\|x(0)\|\leq \alpha_1, \|x_0(t)\|\leq \alpha_2$ for all $t\geq 0$, and
$|d_{i}^{(l)}(t)|\leq \alpha_{3i}$, $|w_{i}^{(j)}(t)|\leq \alpha_{3i}$
for all $t\geq 0$, where $i=1,\cdots,m,l=0,\cdots,n,j=0,\cdots,n+1$.
\end{assumption}
\begin{remark}
The following three aspects are involved in Assumption \ref{as3}. Firstly,
the initial value should be in a given compact set, indicating that
the practical output consensus problem for the leader-follower MASs (\ref{s1})-(\ref{s4}) can only be solved in the semi-global sense.
Secondly, the state and control input of the leader are required to be bounded.
Finally, the boundedness of derivatives of external disturbances and measurement noises in Assumption \ref{as3} is
by reason of the fact that the actual total disturbance \dref{totaldisturbance} of each agent is to be estimated and compensated in the closed-loop.
\end{remark}

Based on the transformation \dref{s2} and Assumptions \ref{a22}-\ref{as3}, it can be easily obtained  that $\|\bar{x}(0)\|\leq \alpha_4$
for some positive constant $\alpha_4$.

Set
\begin{eqnarray}
V_{1i}(\varrho_{i})=W_i\varrho_{i}^{\top}P\varrho_{i},\; \forall \varrho_{i}\in \mathbb{R}^{n},\; i=1,\cdots,m,
\end{eqnarray}
with $P$ be specified in \dref{p1}, and it can be easily obtained that
\begin{equation}
\begin{split}
&V_1(\varrho)\triangleq \varrho^{\top}(W\otimes P)\varrho=\sum_{i=1}^mV_{1i}(\varrho_{i}),\cr &\forall \varrho=(\varrho^{\top}_{1},\cdots,\varrho^{\top}_{m})^{\top}\in \mathbb{R}^{nm}.
\end{split}
\end{equation}
In addition, since $V_{1}$ is a radially unbounded continuous function, we can define two compact sets in $\mathbb{R}^{nm}$ as follows:
\begin{equation}\label{compatts}
\begin{split}
&\Omega_{1}=\{z\in \mathbb{R}^{nm}: V_{1}(z)\leq \max_{s\in \mathbb{R}^{nm},\|s\|\leq \alpha_1+\alpha_2+\alpha_4}V_{1}(s)+1\},\\
&\Omega_{2}=\{z\in \mathbb{R}^{nm}: V_{1}(z)\leq \max_{s\in \mathbb{R}^{nm},\|s\|\leq \alpha_1+\alpha_2+\alpha_4}V_{1}(s)\}.\\
\end{split}
\end{equation}
The constant $M$ in the consensus protocols \dref{con} can be specified as
\begin{equation}\label{Mconstant}
M\triangleq \|L_1\otimes K \|\cdot\sup\{\|z\|:z\in \Omega_{1}\}.
\end{equation}

Define
\begin{equation}\label{comfpat3}
\Omega_{3}=\{z\in \mathbb{R}^{nm}: \|z\|\leq \frac{M}{\|L_1\otimes K \|}+\|{1}_m\otimes \mathbb{I}_n\|\alpha_{2}\},
\end{equation}
and for $i=1,\cdots,m$, we define
\begin{equation}
\Omega_{3i}=\{z_{i}\in \mathbb{R}^{n}: z=(z^{\top}_{1},\cdots,z^{\top}_{m})^{\top}\in \Omega_{3}\}.
\end{equation}
Set $\mathcal{A}_{i}=[-\alpha_{3i},\alpha_{3i}]$, and
\begin{equation}\label{qqiii}
\begin{split}
&\hspace{-0.2cm}Q_{i}=\sup_{(\bar{x}_{i1},\cdots,\bar{x}_{in},d_{i},\cdots,d_{i}^{(n-2)},w_{i},\cdots,w_{i}^{(n-1)})^{\top}\in \Omega_{3i}\times \mathcal{A}^{2n-1}_{i}}|\psi_{i}(\bar{x}_{i1},\cr&\hspace{0.4cm}\cdots,
\bar{x}_{in},
d_{i},\cdots,d_{i}^{(n-2)},w_{i},\cdots,w_{i}^{(n-1)})|,\;i=1,\cdots,m.
\end{split}
\end{equation}
Moreover, we define
\begin{eqnarray}\label{compat344}
&&\Omega_{4i}=\{z_{i}\in \mathbb{R}^{n}: \|z_{i}\|\leq Q_{i}\},\; i=1,\cdots,m,\cr &&
\Omega_{4}=\underbrace{\Omega_{41}\times\cdots\times\Omega_{4m}}_{m},
\end{eqnarray}
and then the constants $N_{i}$'s in the consensus protocols \dref{con} are specified as
\begin{equation}\label{Nconstant}
\begin{split}
&\hspace{-0.2cm}N_{i}=\sup_{(x_{i1},\cdots,x_{in},d_{i},\cdots,d_{i}^{(n-1)})^{\top}\in \Omega_{4i}\times \mathcal{A}^{n}_{i}}|\varphi_{i}(x_{i1},\cdots,
x_{in},\cr&\hspace{0.6cm}d_{i},\cdots,d_{i}^{(n-1)})|+\alpha_{3i},\;i=1,\cdots,m.
\end{split}
\end{equation}

The practical output consensus and boundedness of leader-follower MASs (\ref{s1})-(\ref{s4}) under the ADRC consensus protocols (\ref{con})
are summarized up as the following theorem.

\begin{theorem}\label{tt1}
 Suppose that Assumptions \ref{a1}-\ref{as3} hold, then the leader-follower MASs (\ref{s1})-(\ref{s4}) under the ADRC control protocols (\ref{con})
can achieve the practical output consensus, and there is an $r$-independent positive constant $\Lambda$, such that $\|x(t)\|\leq \Lambda$ for all $t\geq 0$.
\end{theorem}

\begin{proof}
 For $i=1,\cdots,m$, we set
\begin{equation}\label{s5}
\begin{cases}
\eta_{ij}=r^{n+1-j}(\bar{x}_{ij}-\hat{\bar{x}}_{ij}),\;j=1,\cdots,n+1,\\
\eta_i=(\eta_{i1},\cdots,\eta_{i,n+1})^\top,
\eta=(\eta_1^{\top},\cdots,\eta_m^{\top})^\top,\\
\hat{\bar{x}}_{0j}=x_{0j},\;k=1,\cdots,n,\\
\varrho_{ij}=\bar{x}_{ij}-x_{0j}, \hat{\varrho}_{ij}=\hat{\bar{x}}_{ij}-\hat{x}_{0j},\;j=1,\cdots,n,\\
\varrho_{i}=(\varrho_{i1},\cdots,\varrho_{in})^\top,\hat{\varrho}_{i}=(\hat{\varrho}_{i1},\cdots,\hat{\varrho}_{in})^\top,\\
\varrho=(\varrho^{\top}_1,\cdots, \varrho^{\top}_m)^\top, \hat{\varrho}=(\hat{\varrho}_1^\top,\cdots,\hat{\varrho}_m^\top)^\top,\\
\hat{\bar{x}}_i=(\hat{\bar{x}}_{i1},\cdots,\hat{\bar{x}}_{in})^\top,
\hat{\bar{x}}=(\hat{\bar{x}}_1^\top,\cdots,\hat{\bar{x}}_m^\top)^\top,\\
\bar{u}_i(t)=K\mathfrak{\vartheta}_i(t)-\hat{\bar{x}}_{i,n+1}(t).
\end{cases}
\end{equation}

 There is a unique positive definite matrix $G\in \mathbb{R}^{(n+1)\times (n+1)}$ such that
\begin{equation}\label{lyapunovequation}
UG+GU^\top=-\mathbb{I}_{n+1},
\end{equation}
where $U$ is the Hurwitz matrix specified in \dref{umatrix}.

Set
\begin{eqnarray}\label{V1V2defintion}
 V_{2i}(\eta_{i})=\eta^{\top}_{i}G\eta_{i}, \; \forall \eta_{i}\in \mathbb{R}^{n+1}, i=1,\cdots,m,
\end{eqnarray}
and it can be simply obtained that
\begin{equation}
\begin{split}
V_2(\eta)\triangleq \eta^{\top}(\mathbb{I}_m\otimes G)\eta=\sum_{i=1}^mV_{2i}(\eta_{i}), \forall \eta\in \mathbb{R}^{nm+m}.
\end{split}
\end{equation}

We can show that $\varrho_{i}(t)$'s and $\eta_{i}(t)$'s satisfy
\begin{equation}\label{s6d}
\begin{cases}
\dot{\varrho}_{i1}(t)=\varrho_{i2}(t),\cr
\hspace{1.2cm}  \vdots \cr
\dot{\varrho}_{i,n-1}(t)=\varrho_{in}(t),\\
\dot{\varrho}_{in}(t)=\bar{x}_{i,n+1}(t)+u_i(t)-u_0(t),\\
\dot{\eta}_{i1}(t)=r(\eta_{i2}(t)-k_1\eta_{i1}(t)),\cr
\hspace{1.2cm}  \vdots \cr
\dot{\eta}_{in}(t)=r(\eta_{i,n+1}(t)-k_n\eta_{i1}(t)),\\
\dot{\eta}_{i,n+1}(t)=-rk_{n+1}\eta_{i1}(t)+\dot{\bar{x}}_{i,n+1}(t),\; i=1,\cdots,m.
\end{cases}
\end{equation}
Moreover, by some mathematical operations, the compact form of \dref{s6d} can be obtained as

\begin{equation}\label{s8}
\begin{cases}
\dot{\varrho}(t)=(\mathbb{I}_m\otimes A)\varrho(t)+(\mathbb{I}_m \otimes B)\Xi(t)\\
\hspace{0.65cm}=(\mathbb{I}_m\otimes A)\varrho(t){+}(L_1\otimes BK)\hat{\varrho}(t)
+(\mathbb{I}_m\otimes B)\bigtriangleup(t),\\
\dot{\eta}(t)=r(\mathbb{I}_m\otimes U)\eta(t)+(\mathbb{I}_m\otimes B_{n+1})\bigtriangledown(t),
\end{cases}
\end{equation}
where
\begin{eqnarray*}
&&\Xi(t)=(\bar{x}_{1,n+1}(t)+u_1(t)-u_0(t),\cdots,\bar{x}_{m,n+1}(t)\cr&&\hspace{1.1cm}+u_m(t)-u_0(t))^\top,
\end{eqnarray*}
\begin{equation*}
\begin{split}
&\triangle(t)=(\bar{x}_{1,n+1}(t)-\hat{\bar{x}}_{1,n+1}(t)-\bar{u}_1(t)+u_1(t)-u_0(t),\\& \cdots,
\bar{x}_{m,n+1}(t)-\hat{\bar{x}}_{m,n+1}(t)-\bar{u}_m(t)+u_m(t)-u_0(t))^\top,
\end{split}
\end{equation*}
\begin{equation*}
 B_{n+1}=(0,\cdots,0,1)^{\top}\in\mathbb{R}^{n+1},
\end{equation*}
\begin{equation*}
\bigtriangledown(t)=(\dot{\bar{x}}_{1,n+1}(t),\cdots,\dot{\bar{x}}_{m,n+1}(t))^\top.
\end{equation*}
We proceed the proof by the following three steps.

{\bf Step 1: It is proved that there exists $r_1>0$ such that for any $r\geq r_{1}$, there holds $\varrho(t)\in \Omega_{1}$ for all $t\geq 0$, which also
indicates that we can find an $r$-independent constant $\Lambda$ to make $\|x(t)\|\leq \Lambda$, $\forall t\geq0$.}

We start the proof of Step 1. By the definition of compact set $\Omega_{2}$ in \dref{compatts}, we have $\varrho(0)\in \Omega_{2}$ and $\varrho(0)$ is an interior point of $\Omega_{2}$.
 Thus, by the continuity of $\varrho(t)$ with respect to $t$, $\varrho(t)$ will stay in $\Omega_{2}\subset\Omega_{1}$ in a short period of time from $t=0$.
 Since $\bar{x}(t)=\varrho(t)+({1}_m\otimes \mathbb{I}_n)x_0(t)$ and $x_0(t)$ is bounded by Assumption \ref{as3},
  $\bar{x}(t)$ will lie in $\Omega_{3}$ defined in \dref{comfpat3} in a short period of time from $t=0$.
  It then follows from Assumption \ref{as3}, \dref{26equnation}, and \dref{qqiii} that
  $x(t)\in \Omega_{4}$ for $\Omega_{4}$ defined in \dref{compat344} in a short period of time from $t=0$.
  Furthermore, by the equivalent expression of total disturbances $\bar{x}_{i,n+1}(t)$'s in \dref{exprssiondisutranve} and Assumptions \ref{a22}-\ref{as3},
 we have $\bar{x}_{i,n+1}(t)\leq N_{i}$  within a short time from $t=0$ for $N_{i}$'s given in \dref{Nconstant} independent of $r$.
By the $\varrho$-subsystem of \dref{s6d} and the boundedness of $u_{i}(t)-u_{0}(t)$ guaranteed by the saturation functions whose bounds are also independent of $r$,
it can be concluded that there is an $r$-independent time $t_{0}>0$, such that $\varrho(t)\in \Omega_{1},\; \forall t\in [0, t_{0}]$.

The conclusion of Step 1 can be obtained by the following reductio ad absurdum. Let us first assume that the conclusion of Step 1 is false, and then on the basis of the continuity of $\varrho(t)$ in $t$, there is $r$-dependent constants $t_{1},t_{2}$ satisfying $t_{2}>t_{1}\geq t_{0}$ such that
\begin{equation}\label{comptactsdf}
\begin{split}
&\varrho(t_{1})\in \partial \Omega_{2},\; \varrho(t_{2})\in\partial \Omega_{1},\\
&\{\varrho(t):t\in(t_{1},t_{2}]\}\subset \Omega_{1}-\Omega_{2}^0,\\
&\{\varrho(t):t\in[0,t_{2}]\}\subset\Omega_{1},
\end{split}
\end{equation}
where $\partial \Omega_{j}\;(j=1,2)$ and $\Omega_{2}^0$  represent the boundary of $\Omega_{j}$ and the interior of $\Omega_{2}$, respectively.

For $i=1,\cdots,m$, computing the derivative of the total disturbance $\bar{x}_{i,n+1}(t)$ with regard to $t$, it is obtained that
\begin{equation}\label{s41}
\begin{split}
&\dot{\bar{x}}_{i,n+1}(t)=h^{(1)}_{in}(x_i(t),d_{i}(t))\cr&+\sum_{l=1}^{n-1}h^{(l+1)}_{i,n-l}(x_{i1}(t),\cdots,x_{i,n-l}(t),d_{i}(t))+w^{(n+1)}_{i}(t).
\end{split}
\end{equation}
Similar to \dref{computations}, it can be obtained that there exist continuous functions $f_{i}$ and $g_{i}$ such that
\begin{equation}
\begin{split}
& h^{(1)}_{in}(x_i(t),d_{i}(t))=f_{i}(x_{i}(t),d_{i}(t),\dot{d}_{i}(t),u_{i}(t)), \cr&
h^{(l+1)}_{i,n-l}(x_{i1}(t),\cdots,x_{i,n-l}(t),d_{i}(t)) \cr&
=g_{i}(x_{i}(t),d_{i}(t),\dot{d}_{i}(t),\cdots,d^{(l+1)}_{i}(t),u_{i}(t))
\end{split}
\end{equation}

Similar to aforementioned deductions, since $\varrho(t)$ is bounded in $t\in[0,t_2]$, $\bar{x}(t)$ and
then $x(t)$ is bounded in $t\in[0,t_2]$.
In addition, $u_{i}(t)$'s are bounded ensured by the saturation functions.
These together with Assumption \ref{as3}, yield that
\begin{equation}\label{s44}
|\dot{\bar{x}}_{i,n+1}(t)|\leq D_{1i},\; \forall t\in [0,t_{2}],
\end{equation}
for some positive constants $D_{1i}$ independent of $r$.
According to the definition of $V_{2i}$ in \dref{V1V2defintion},
we have
\begin{eqnarray}
&&\lambda_{\min}(G)\|\eta_{i}\|^{2}\leq V_{2i}(\eta_i)\leq \lambda_{\max}(G)\|\eta_{i}\|^{2}, \cr&&
\frac{\partial V_{2i}(\eta_i)}{\partial \eta_{i,n+1}}\leq 2\lambda_{\max}(G)\|\eta_i\|,\;\forall \eta_{i}\in \mathbb{R}^{n+1}.
\end{eqnarray}
These together with  \dref{lyapunovequation}, for all $t\in [0,t_{2}]$, follow that
\begin{eqnarray}\label{s45}
&&\hspace{-1.5cm}\frac{dV_{2i}(\eta_i(t))}{dt}=r\sum_{j=1}^n\frac{\partial V_{2i}(\eta_i(t))}{\partial \eta_{ij}}\left[\eta_{i,j+1}(t)-k_j\eta_{i1}(t)\right]\cr &&\hspace{-1.5cm}
-r\frac{\partial V_{2i}(\eta_i(t))}{\partial \eta_{i,n+1}}k_{n+1}\eta_{i1}(t)+\frac{\partial V_{2i}(\eta_i(t))}{\partial \eta_{i,n+1}}\dot{\bar{x}}_{i,n+1}(t)\cr && \hspace{-1.5cm}\leq -\frac{r}{\lambda_{\max}(G)}V_{2i}(\eta_i(t))+\frac{2\lambda_{\max}(G)D_{1i}}{\sqrt{\lambda_{\min}(G)}}\sqrt{V_{2i}(\eta_i(t))}.
\end{eqnarray}
According to the common-used inequality $\disp(\sum^{m}_{i=1}a_{i})^{p}\leq m^{p-1}\sum^{m}_{i=1}a^{p}_{i}$ for any $a_{i}\geq 0$ and $p>1$, we have
\begin{equation*}
\sum_{i=1}^m\sqrt{V_{2i}(\eta_i)}\leq \sqrt{m}\sqrt{V_2(\eta)}.
\end{equation*}

Therefore, for $t\in[t_{1},t_2]$, we have
\begin{equation*}
\begin{split}
&\frac{dV_2(\eta(t))}{dt}=\sum_{i=1}^m\frac{dV_{2i}(\eta_i(t))}{dt}\\&
\leq -\frac{r}{\lambda_{\max}(G)}V_2(\eta(t))+\frac{2\sqrt{m}\lambda_{\max}(G)\disp\max_{1\leq i\leq m}D_{1i}}{\sqrt{\lambda_{\min}(G)}}\sqrt{V_2(\eta(t))},
\end{split}
\end{equation*}
which means
\begin{eqnarray*}
&&\sqrt{V_2(\eta(t))}\cr&&\leq e^{-\frac{r}{2\lambda_{\max}(G)}t}\sqrt{V_2(\eta(0))}+\frac{2\sqrt{m}\lambda^{2}_{\max}(G)\disp\max_{1\leq i\leq m}D_{1i}}{\sqrt{\lambda_{\min}(G)}r}.
\end{eqnarray*}
It can be obtained that
\begin{eqnarray}\label{conductionsf}
&&\hspace{-0.7cm}e^{-\frac{rt}{2\lambda_{\max}(G)}}\sqrt{V_2(\eta(0))}\cr&&\hspace{-0.7cm}\leq e^{-\frac{rt_{0}}{2\lambda_{\max}(G)}}\sqrt{V_2(\eta(0))}\leq
e^{-\frac{rt_{0}}{2\lambda_{\max}(G)}}\sqrt{\lambda_{\max}(G)}\|\eta(0)\|\cr&&\hspace{-0.7cm}
\leq e^{-\frac{rt_{0}}{2\lambda_{\max}(G)}}\sqrt{\lambda_{\max}(G)}[\sum^{m}_{i=1}\sum^{n+1}_{j=1}r^{2n+2-2j}(\bar{x}_{ij}(0)-\hat{\bar{x}}_{ij}(0))^{2}]^{\frac{1}{2}}
\cr&&\hspace{-0.7cm} \rightarrow 0 \;\; \mbox{in}\;\; t\in[t_1,t_2] \;\;\mbox{as} \;\; r\rightarrow+\infty,
\end{eqnarray}
where $t_{0}>0$ is a  constant independent of $r$.
Thus,
\begin{equation*}
\|\eta(t)\|\rightarrow 0 \;\; \mbox{in}\;\; \mbox{uniformly in } t\in[t_1,t_2] \;\;\mbox{as} \;\; r\rightarrow+\infty.
\end{equation*}

This yields, for
\begin{equation}\label{defingef}
 \varsigma\triangleq\min\{\frac{1}{2},\frac{1}{2\|L_1\otimes K\|},\frac{2}{\|L_1\otimes K\|^{2}},\frac{\beta_{1}\disp\min_{\varrho \in \Omega_{1}}V_{1}(\varrho)}{\beta_{2}}\},
\end{equation}
with $\beta_{1},\beta_{2}$ be specified in \dref{bspecifeied},
 there exists
 $r_1\geq 1$ such that for any $r\geq r_{1}$, we have $\|\eta(t)\|\leq \varsigma,\;\forall t\in[t_1,t_2]$.
By a direct compuation, we have
\begin{equation}\label{s46}
\begin{split}
&\|\varrho(t)-\hat{\varrho}(t)\|=\|\bar{x}(t)-\hat{\bar{x}}(t)\|\\
&=\|(\frac{\eta_{11}(t)}{r^n},\cdots,\frac{\eta_{1n}(t)}{r},\cdots,\frac{\eta_{m1}(t)}{r^n},\cdots,\frac{\eta_{mn}(t)}{r})\|\\
&\leq\|\eta(t)\|,\; \forall t\geq 0.
\end{split}
\end{equation}
Since $\varrho(t)\in \Omega_{1}$ for $t\in [0,t_{2}]$, by the definition of $M$ in \dref{Mconstant}, we further have
\begin{equation}\label{s47}
\begin{split}
&\|L_1\otimes K \hat{\varrho}(t)\|\leq \|L_1\otimes K\|\left(\|\varrho(t)\|+\|\hat{\varrho}(t)-\varrho(t)\|\right)\\
&\hspace{2.1cm}\leq M+\frac{1}{2},\; \forall t\in[t_{1},t_{2}].
\end{split}
\end{equation}
By the definition of (\ref{ens}), it follows that $|K\vartheta_i(t)|\leq \|L_1\otimes K \hat{\varrho}(t)\|\leq M+\frac{1}{2},\forall t\in[t_{1},t_{2}]$.
Therefore, for $i=1,\cdots,m$, if $|K\vartheta_i(t)|\leq M$, it is directly obtained that
\begin{equation*}
K\vartheta_i(t)-\mathrm{sat}_{M}(K\vartheta_i(t))=0,\; \forall t\geq 0.
\end{equation*}
If $K\vartheta_i(t)>M$, then
\begin{equation}
\begin{split}
&\hspace{-0.6cm}|K\vartheta_{i}(t)-M|\leq K\vartheta_{i}(t)-\|L_1\otimes K \varrho(t)\|\cr&\hspace{-0.6cm}
\leq \|L_1\otimes K \|\cdot\|\hat{\varrho}(t)-\varrho(t)\|\leq \|L_1\otimes K \|\varsigma,\; \forall t\in[t_{1},t_{2}],
\end{split}
\end{equation}
and
\begin{equation*}
\begin{split}
&~~|K\vartheta_i(t)-\mathrm{sat}_{M}(K\vartheta_i(t))|\\
&=|K\vartheta_i(t)+\frac{1}{2}(K\vartheta_i(t))^2-(M+1)K\vartheta_i(t)+\frac{1}{2}M^2|\\
&=\frac{(K\vartheta_i(t)-M)^2}{2}\leq \varsigma,\; \forall t\in[t_{1},t_{2}].
\end{split}
\end{equation*}
The similar conclusion can be directly drew for $K\vartheta_i(t)<-M$  by the fact that $\mathrm{sat}_M(\cdot)$ is an odd function.
Hence, $|K\vartheta_i(t)-\mathrm{sat}_{M}(K\vartheta_i(t))|\leq\varsigma$ for all $t\in[t_1,t_2]$.
Since $\varrho(t)\in \Omega_{1}$ for $t\in [0,t_{2}]$, exactly following the reasoning of the fore of Step 1,
it can be obtained that  $|\bar{x}_{i,n+1}(t)|\leq N_{i}$ for $t\in [0,t_{2}]$ and $i=1,\cdots,m$.
Since $|\hat{\bar{x}}_{i,n+1}(t)|\leq |\bar{x}_{i,n+1}(t)|+\|\eta(t)\|\leq N_{i}+\frac{1}{2}$ for all $t\in[t_1,t_2]$,
it can be similarly concluded that $|\hat{\bar{x}}_{i,n+1}(t)-\mathrm{sat}_{N_{i}}(\hat{\bar{x}}_{i,n+1}(t))|\leq\varsigma$ and
then $|\bar{x}_{i,n+1}(t)-\mathrm{sat}_{N_{i}}(\hat{\bar{x}}_{i,n+1}(t))|\leq 2\varsigma$ for all $t\in[t_1,t_2]$.

So, we have
\begin{equation}\label{errorer}
\|\bigtriangleup(t)\|\leq 3\sqrt{m}\varsigma,\;\forall t\in[t_1,t_2] .
\end{equation}
Noting that the positive definite matrix $P\in \mathbb{R}^{n\times n}$ solves the Riccati equation \dref{p1} and the small constant $\varsigma$
is defined in \dref{defingef}, and taking the derivative of $V_1(\varrho(t))$ with regard to $t$ along $\varrho$-subsystem of (\ref{s8}) to obtain
\begin{eqnarray}\label{s10}
&&\disp\frac{dV_1(\varrho(t))}{dt}=\varrho^{\top}(t)(W\otimes (PA+A^{\top}P)\cr&&\disp
-(WL_1+L^{\top}_1W)\otimes PBB^{\top}P)\varrho(t)\cr&&\disp
-2\varrho^{\top}(t)(WL_1\otimes PBB^{\top}P)(\hat{\varrho}(t)-\varrho(t))\cr&&\disp
+2\varrho^{\top}(t)(W\otimes PB)\bigtriangleup(t)\cr&&\disp
\leq \varrho^{\top}(t)(W\otimes (PA+A^{\top}P)-(\mu\mathbb{I}_m\otimes PBB^{\top}P))\varrho(t)\cr&&\disp
-2\varrho^{\top}(t)(WL_1\otimes PBB^{\top}P)(\hat{\varrho}(t)-\varrho(t))\cr&&\disp
+2\varrho^{\top}(t)(W\otimes PB)\bigtriangleup(t)\cr&&\disp
\leq \underline{\varrho}^{\top}(t)(\mathbb{I}_m\otimes(PA+A^{\top}P-\mu_0PBB^{\top}P))\underline{\varrho}(t)\cr&&\disp
-2\varrho^{\top}(t)(WL_1\otimes PBB^{\top}P)(\hat{\varrho}(t)-\varrho(t))\cr&&\disp
+2\varrho^{\top}(t)(W\otimes PB)\bigtriangleup(t)\cr&&\disp
= -\underline{\varrho}^{\top}(t)\underline{\varrho}(t)-2\varrho^{\top}(t)(WL_1\otimes PBB^{\top}P)(\hat{\varrho}(t)-\varrho(t))\cr&&\disp
+2\varrho^{\top}(t)(W\otimes PB)\bigtriangleup(t)\cr&&\disp
=-\varrho^{\top}(t)(W\otimes \mathbb{I}_n)\varrho(t)\cr&&\disp
-2\varrho^{\top}(t)(WL_1\otimes PBB^{\top}P)(\hat{\varrho}(t)-\varrho(t))\cr&&\disp
+2\varrho^{\top}(t)(W\otimes PB)\bigtriangleup(t)\cr&&\disp
\leq -\beta_1V_1(\varrho(t))+\beta_2\varsigma<0,~~t\in[t_1,t_2],
\end{eqnarray}
where
\begin{eqnarray}\label{bspecifeied}
&&\hspace{-1.4cm}  \mu=\lambda_{\min}(WL_1+L^{\top}_1W),\;\mu_{0}=\mu\lambda_{\min}(W^{-1}),\cr
&&\hspace{-1.4cm} \underline{\varrho}(t)=(W^{\frac{1}{2}}\otimes \mathbb{I}_n)\varrho(t),\;\;
\beta_1=\frac{\lambda_{\min}(W)}{\lambda_{\max}(W\otimes P)},\cr&&\hspace{-1.4cm}
\beta_{2}=2M\|WL_1\otimes PBB^{\top}P\|+6M\|W\otimes PB\|\sqrt{m}.
\end{eqnarray}
It follows from \dref{s10} that $V_1(\varrho(t))$ is monotonic decreasing in $t\in[t_1,t_2]$. However, by (\ref{compatts}) and (\ref{comptactsdf}),
we have $V_1(\varrho(t_2))=V_1(\varrho(t_1))+1$, which leads to the contradiction. Consequently, $\{\varrho(t):t\in[0,\infty)\}\subset \Omega_1$ for any $r\geq r_1$,
and then  there is an $r$-independent constant $\Lambda$, with the result that $\|x(t)\|\leq \Lambda, \forall t\geq0$.

{\bf Step 2: It is proved that for any $T>0$, $\|\eta(t)\|\rightarrow 0$ uniformly in $t\in[T,+\infty)$ as $r\rightarrow+\infty$, which also means $|\bar{x}_{ij}(t)-\hat{\bar{x}}_{ij}(t)|\rightarrow0$ uniformly in $t\in[T,+\infty)$ as $r\rightarrow+\infty$ for  $i=1,\cdots,m, j=1,\cdots,n+1$.}

Similar to the proof of Step 1 and $\{\varrho(t):t\in[0,+\infty)\}\subset \Omega_1$ for any $r\geq r_1$, it can be concluded that there are $r$-independent positive constants $D_{2i}$ satisfying
\begin{equation*}
|\dot{\bar{x}}_{i,n+1}(t)|\leq D_{2i},\;\forall t\geq 0,
\end{equation*}
and then
\begin{eqnarray*}
&&\sqrt{V_2(\eta(t))}\cr&&\leq e^{-\frac{r}{2\lambda_{\max}(G)}t}\sqrt{V_2(\eta(0))}+\frac{2\sqrt{m}\lambda^{2}_{\max}(G)\disp\max_{1\leq i\leq m}D_{2i}}{\sqrt{\lambda_{\min}(G)}r}
\end{eqnarray*}
for all $t\geq 0$. Similar to \dref{conductionsf}, we have
\begin{eqnarray}\label{conductsf}
e^{-\frac{r}{2\lambda_{\max}(G)}t}\sqrt{V_2(\eta(0))}\rightarrow 0
\end{eqnarray}
uniformly in $t\in[T,+\infty]$ as $r\rightarrow +\infty$.
These lead to $\|\eta(t)\|\rightarrow 0$ uniformly in $t\in[T,+\infty)$ as $r\rightarrow+\infty$, which ends the proof of Step 2.

{\bf Step 3: It is proved that for any $\varepsilon>0$, there is $r^*>0$ such that
for any $r\geq r^*$ and all $t\geq t_{r}$ with $t_{r}$ be an
$r$-dependent positive constant, there holds $\|\varrho(t)\|\leq \varepsilon$ and then
 $|y_i(t)-y_{0}(t)|\leq \varepsilon$ for all $t\geq t_{r}$ and $i=1,\cdots,m.$}

By the conclusion of Step 2, similar to the proof of \dref{errorer}, for any $\varepsilon>0$,
there is $r^{*}\geq r_{1}$ with the result that for any $r\geq r^{*}$ and $T>0$, there holds
\begin{eqnarray}
&&\hspace{-0.7cm}2M\|WL_1\otimes PBB^{\top}P\|\cdot\|\eta(t)\|
+2M\|W\otimes PB\|\cdot\|\Delta(t)\|\cr&&\hspace{-0.7cm}<\beta_{1}\lambda_{\min}(W\otimes P)\varepsilon^{2},\; \forall t\in [T,\infty).
\end{eqnarray}
Therefore, analogue to the proof of (\ref{s10}), it comes to the conclusion that when $\|\varrho(t)\|>\varepsilon$, we have
\begin{eqnarray}\label{s11}
&&\hspace{-1.2cm}\frac{dV_1(\varrho(t))}{dt}\cr&&\hspace{-1.2cm}\leq -\beta_1V_1(\varrho(t))+2M\|WL_1\otimes PBB^{\top}P\|\cdot\|\eta(t)\|\cr&&\hspace{-0.9cm}+2M\|W\otimes PB\|\cdot\|\Delta(t)\|
\cr&&\hspace{-1.2cm}<0.
\end{eqnarray}
This yields that for any $r\geq r^{*}$,
there is an $r$-dependent constant $t_r$ with the result that
\begin{equation*}
\|\varrho(t)\|\leq \varepsilon,\; \forall t\in[t_r,+\infty),
\end{equation*}
which further indicates  that
\begin{equation*}
|y_i(t)-y_{0}(t)|\leq \|\varrho(t)\|\leq\varepsilon,\; \forall t\in[t_r,+\infty).
\end{equation*}
This ends the proof of Theorem \ref{tt1}.
\end{proof}

\begin{remark}
Compared with available literature like \cite{wu2021,misMAS1,misMAS2}, the novelty and the essential difficulty in the theoretical
analysis are brought about by the mismatched uncertainties, but not the mismatched disturbances.
This is the main reason why we do not lay stress on mismatched disturbances in the title of this paper,
which can also be included as part of the mismatched uncertainties. In addition, as a result of
the existence of mismatched disturbances and mismatched uncertainties that are nonvanishing at the steady state,
the consensus can only be addressed with regard to outputs instead of the consensus of other states, in which
other states can only  be guaranteed to be bounded in the ADRC's closed-loop.
\end{remark}
\section{Numerical simulations}\label{se4}

Some numerical simulations are conducted  to bear out
the validity of the ADRC consensus protocols and theoretical result
in this section.
Consider second-order MASs subject to mismatched disturbances,
mismatched uncertainties, and measurement noises as follows:
\begin{equation}\label{ns1}
\begin{cases}
&\dot{x}_{i1}(t)=x_{i2}(t)+h_{i1}(x_{i1}(t),d_{i}(t)),\\
&\dot{x}_{i2}(t)=h_{i2}(x_{i1}(t),x_{i2}(t),d_{i}(t))+u_i(t),\\
&y_i(t)=x_{i1}(t)+w_{i}(t), \; i\in \{1,\cdots,5\},
\end{cases}
\end{equation}
which is a special of  MASs \dref{s1} with $n=2,m=5$, and the leader is described as
a special case of \dref{s4} with $n=2$.
The network topology is shown in Figure 1.
\begin{figure}[htb]
\center{\includegraphics[width=5cm]  {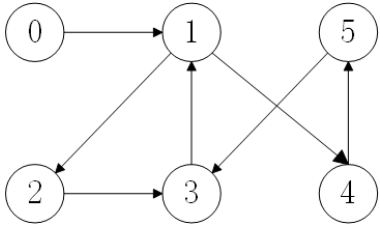}}
 \caption{\label{1} Topology structure}
 \end{figure}

Five ESOs are designed as follows:
\begin{equation}\label{ns2}
\begin{cases}
\dot{\hat{\bar{x}}}_{i1}(t)=\hat{\bar{x}}_{i2}(t)+3r(y_i(t)-\hat{\bar{x}}_{i1}(t)),\\
\dot{\hat{\bar{x}}}_{i2}(t)=\hat{\bar{x}}_{i3}(t)+3r^2(y_i(t)-\hat{\bar{x}}_{i1}(t)),\\
\dot{\hat{\bar{x}}}_{i3}(t)=r^3(y_i(t)-\hat{\bar{x}}_{i1}(t)),\;i=1,\cdots,5,
\end{cases}
\end{equation}
where $k_{i}$'s are chosen as $k_1=k_2=3,k_3=1$ so that the matrix $U$ specified in \dref{umatrix} is Hurwitz.
According to the topology structure and the Riccati equation (\ref{p1}), it can be easily obtained
\begin{equation*}
W^\top=\begin{bmatrix}
5&7&6&2&1
\end{bmatrix}
\end{equation*}
\begin{equation*}
P=\begin{bmatrix}
2.3216& 2.1949\\
2.1949&5.0956
\end{bmatrix}~~ \mathrm{and}~~K=\begin{bmatrix}
 -2.1949&-5.0956
\end{bmatrix}.
\end{equation*}
The ADRC consensus protocols in \dref{con} are designed as
\begin{equation}
u_i(t)=\mathrm{sat}_{5}(K\vartheta_i(t))-\mathrm{sat}_{5}(\hat{\bar{x}}_{i3}(t)),\;i=1,\cdots,5,
\end{equation}
where $\vartheta_i(t)$ and $\mathrm{sat}_{5}(\cdot)$ are defined in  \dref{ens} and \dref{saturaionfunct}, respectively.

In all the numerical simulations, the initial values of system (\ref{ns1}) are selected as $x_1(0)=(0.1,-0.4)^\top$, $x_2(0)=(0.2,0.3)^\top$, $x_3(0)=(0.5,-0.5)^\top$,
$x_4(0)=(0.5,-0.5)^\top$, $x_5(0)=(-0.8,0.7)^\top$, and all initial values of ESOs (\ref{ns2}) are zero. The initial values of the dynamics of the leader is specified as $x_0(0)=[0.3,0.2]^\top$, and its
control input is $u_0(t)=-x_{01}(t)-2x_{02}(t)+\cos(x_{01}^2(t)+x_{02}^2(t)),\;\forall t\geq 0$. It is easy to prove that the states $x_{01}(t),x_{02}(t)$ are bounded.

In Figures \ref{trackingeffect2}-\ref{trackingeffect3},
the mismatched disturbances, mismatched uncertainties, and measurement noises are chosen as follows:
\begin{equation}\label{para1}
\begin{split}
&h_{i1}(x_{i1},d_{i})=0.15e^{x_{i1}}+0.2\cos^{3}(x_{i1})+d^2_{i},\\
&d_{i}(t)=\sqrt{0.3}\sin(2t),\; i=1,2;\\
& h_{i1}(x_{i1},d_{i})=0.2x^{3}_{i1}+0.2x^{2}_{i1}+d_{i},\\
&d_i(t)=0.2\cos(2t),\; i=3,4,5; \\
&h_{i2}(x_{i1},x_{i2},d_{i})=0.3x_{i1}+0.2e^{0.01x_{i2}}+d_{i},\\
&d_i(t)=0.2\sin(t),\;i=1,2;\\
&h_{i2}(x_{i1},x_{i2},d_{i})=0.3x_{i1}+0.2e^{-0.1x_{i2}}+d^{3}_{i},\\
&d_i(t)=\sqrt[3]{0.2}\sin(t),\;i=3,4,5;\\
&w_{i}(t)=\cos(t), i=1,2,\;w_{i}(t)=0.1te^{-t}, i=3,4,5.
\end{split}
\end{equation}
In Figure \ref{trackingeffect2}, the tuning gain $r$ is chosen as $r=10$.
The output consensus effect and the estimation effect of actual total disturbances $\bar{x}_{i3}(t)$
are satisfactory by observing the error curves of $y_{i}(t)-y_{0}(t)$ and $\hat{\bar{x}}_{i3}(t)-\bar{x}_{i3}(t)$,
which can be seen from Figure \ref{trackingeffect2}(a) and Figure \ref{trackingeffect2}(c), respectively.
In Figure \ref{trackingeffect3}, the tuning gain $r$ is increased to be $r=50$.
It can be observed from  Figure \ref{trackingeffect3}(a) and Figure \ref{trackingeffect3}(c)
that both the output consensus effect and the estimation effect of actual total disturbances $\bar{x}_{i3}(t)$
are more satisfactory than those in Figure \ref{trackingeffect2}, which is accord with
the fact indicated by the theoretical result that
the upper bound of the tracking/estimation errors are inverse proportional to the tuning parameter $r$.
The boundedness of the second states $x_{i2}(t)$ of all followers can be seen from  Figure \ref{trackingeffect2}(b) and Figure \ref{trackingeffect3}(b),
which is also consistent with the theoretical result.

In Figure \ref{trackingeffect4}, the tuning gain $r$ is still to be $r=50$, but the
 mismatched disturbances, mismatched uncertainties, and measurement noises are varied as follows:
 \begin{eqnarray}\label{para2}
\begin{split}
&h_{i1}(x_{i1},d_{i})=0.2e^{x_{i1}}+0.3\cos^{3}(x_{i1})+d^{2}_{i},\\
&d_i(t)=\sin(2t),\; i=1,2; \\ &h_{i1}(x_{i1},d_{i})=0.4x^{3}_{i1}+0.4x^{2}_{i1}+d_{i},\\ &d_i(t)=0.3\cos(2t), \;i=3,4,5;\\
&h_{i2}(x_{i1},x_{i2},d_{i})=x_{i1}+0.5e^{0.01x_{i2}}+d_{i},\\
&d_i(t)=\sin(t),\;i=1,2;\\
&h_{i2}(x_{i1},x_{i2},d_{i})=0.4x_{i1}+0.3e^{-0.1x_{i2}}+d^{3}_{i},\\
&d_i(t)=\sqrt[3]{0.5}\sin(t),\;i=3,4,5;\\
&w_{i}(t)=\cos(2t), i=1,2,\;w_{i}(t)=0.2te^{-t}, i=3,4,5.
\end{split}
\end{eqnarray}
Compared with \dref{para1}, although most coefficients in the system functions and disturbances are increased,
it can be observed from Figure \ref{trackingeffect4} that the good outcomes of output consensus, state
boundedness, and estimation of actual total disturbances are still preserved, which demonstrates
the robustness of the proposed ADRC consensus protocols up to a point.
Finally, it can be seen from Figures \ref{trackingeffect2}-\ref{trackingeffect4} that
all the effects of output consensus and estimation of actual total disturbances of
the first two followers are not as good as the others. This is because the system functions in
dynamics of the first two followers are with exponential growth, while the others are only with polynomial growth,
which is consistent with the theoretical result that the tracking/estimation effects are dependent on
the intensity of the disturbances and uncertainties.
\begin{figure}[!htb]\centering
	\subfigure[The trajectories of tracking errors of outputs]
	{\includegraphics[width=4.3cm,height=4cm]{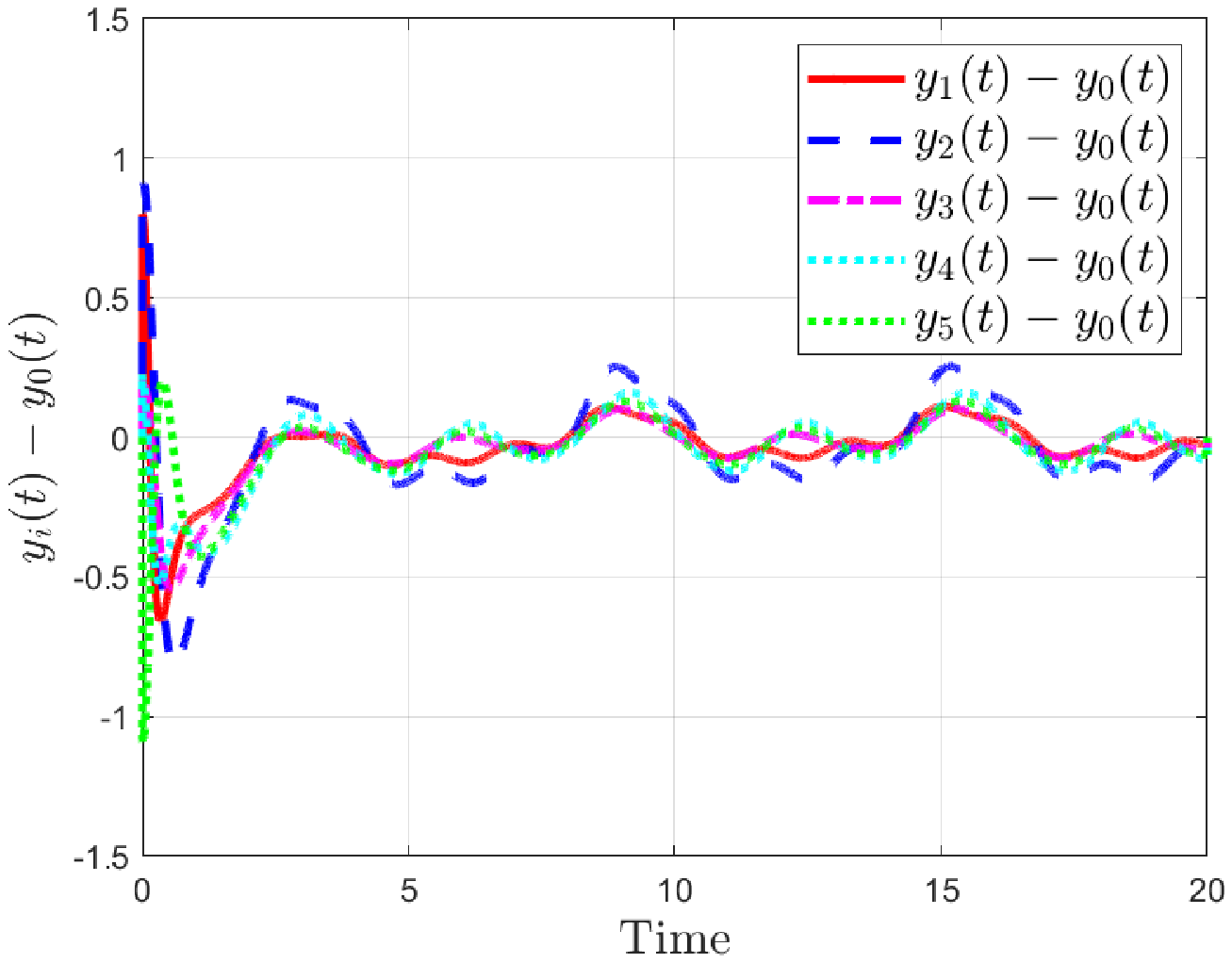}}
	\subfigure[The trajectories of the second state of followers]
	{\includegraphics[width=4.3cm,height=4cm]{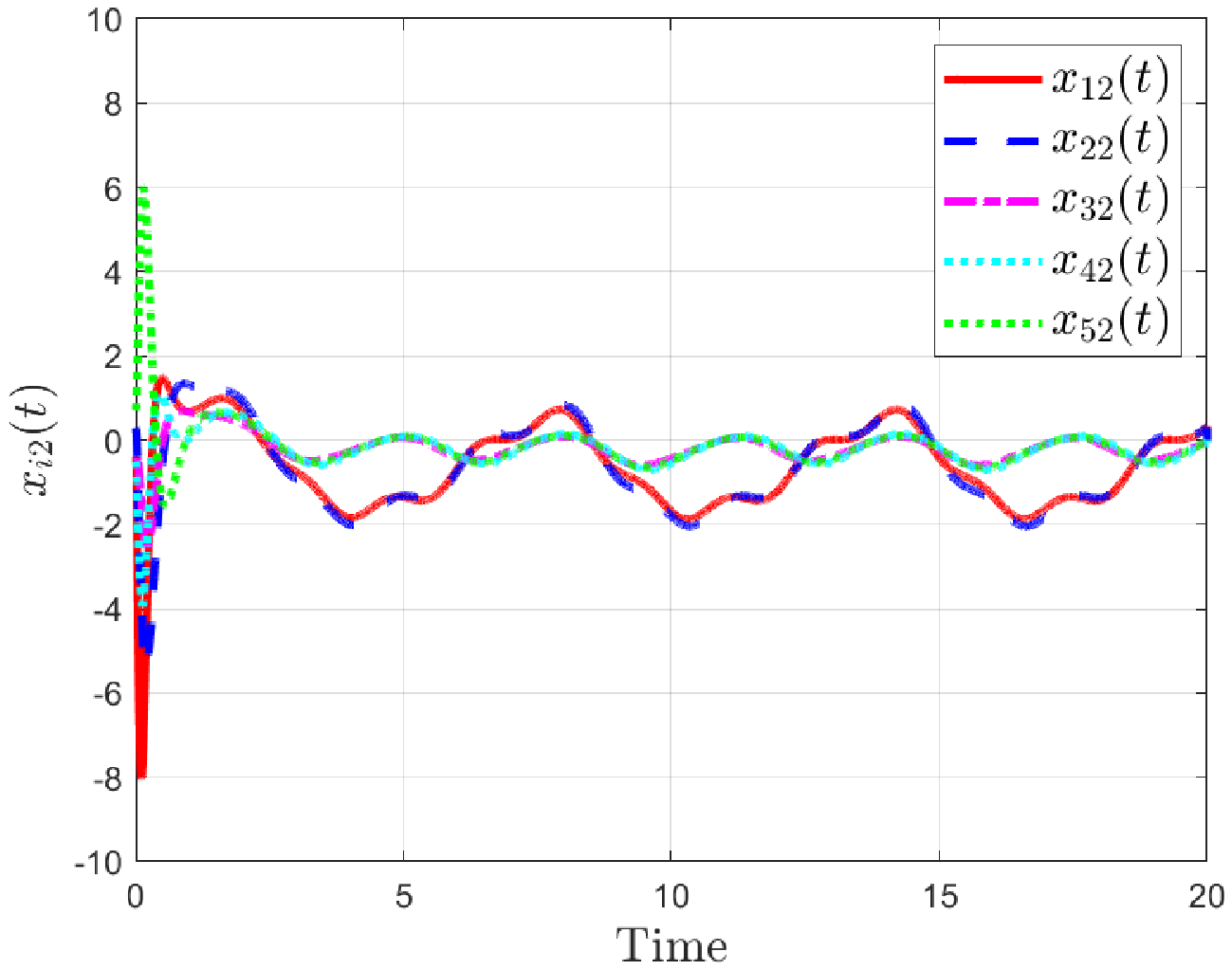}}
	\subfigure[The trajectories of estimation errors of actual total disturbances]
	{\includegraphics[width=4.3cm,height=4cm]{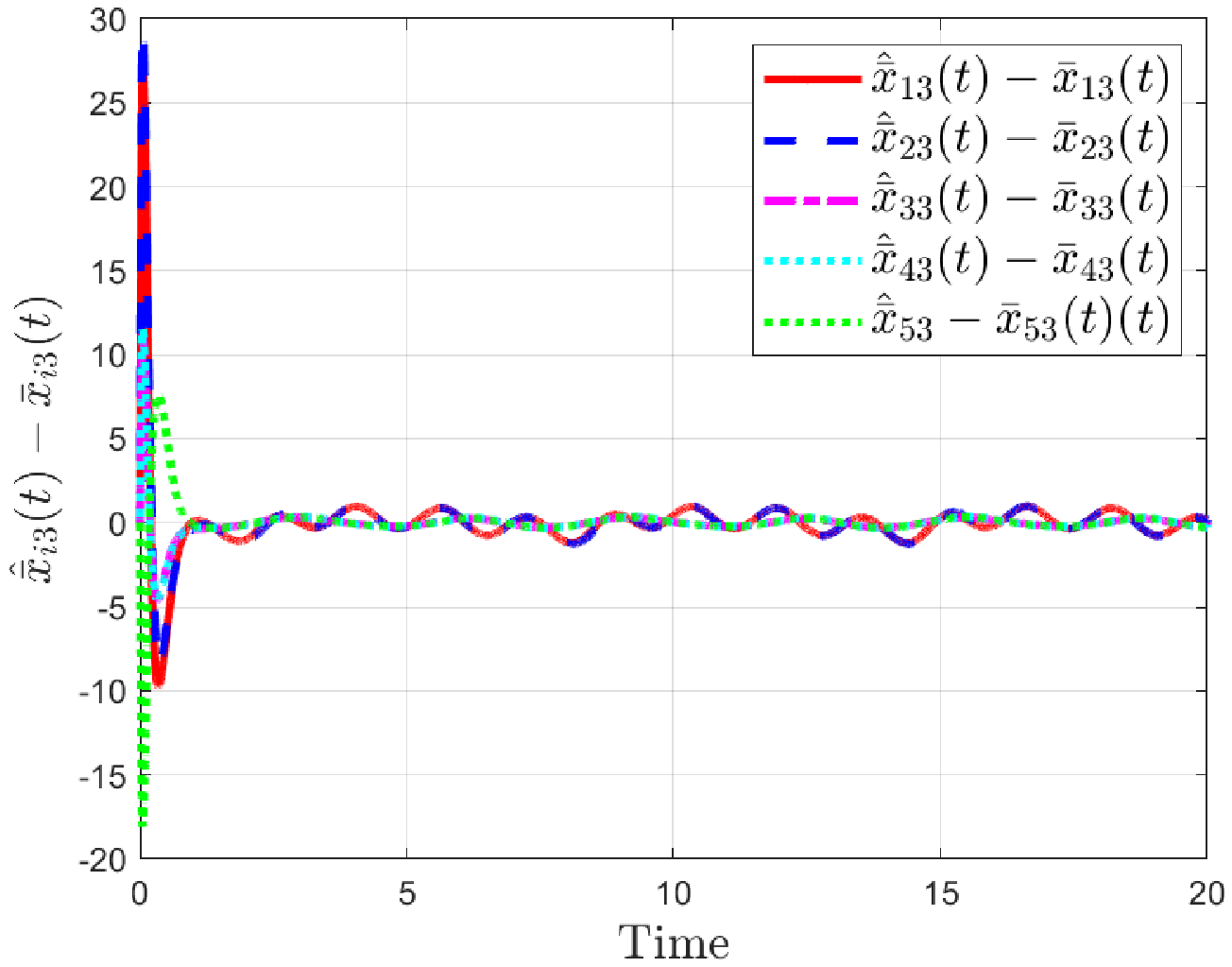}}
	\caption{The effects of output consensus, state bounedness, and estimation of actual total disturbances.}\label{trackingeffect2}
\end{figure}

\begin{figure}[!htb]\centering
	\subfigure[The trajectories of tracking errors of outputs]
	{\includegraphics[width=4.3cm,height=4cm]{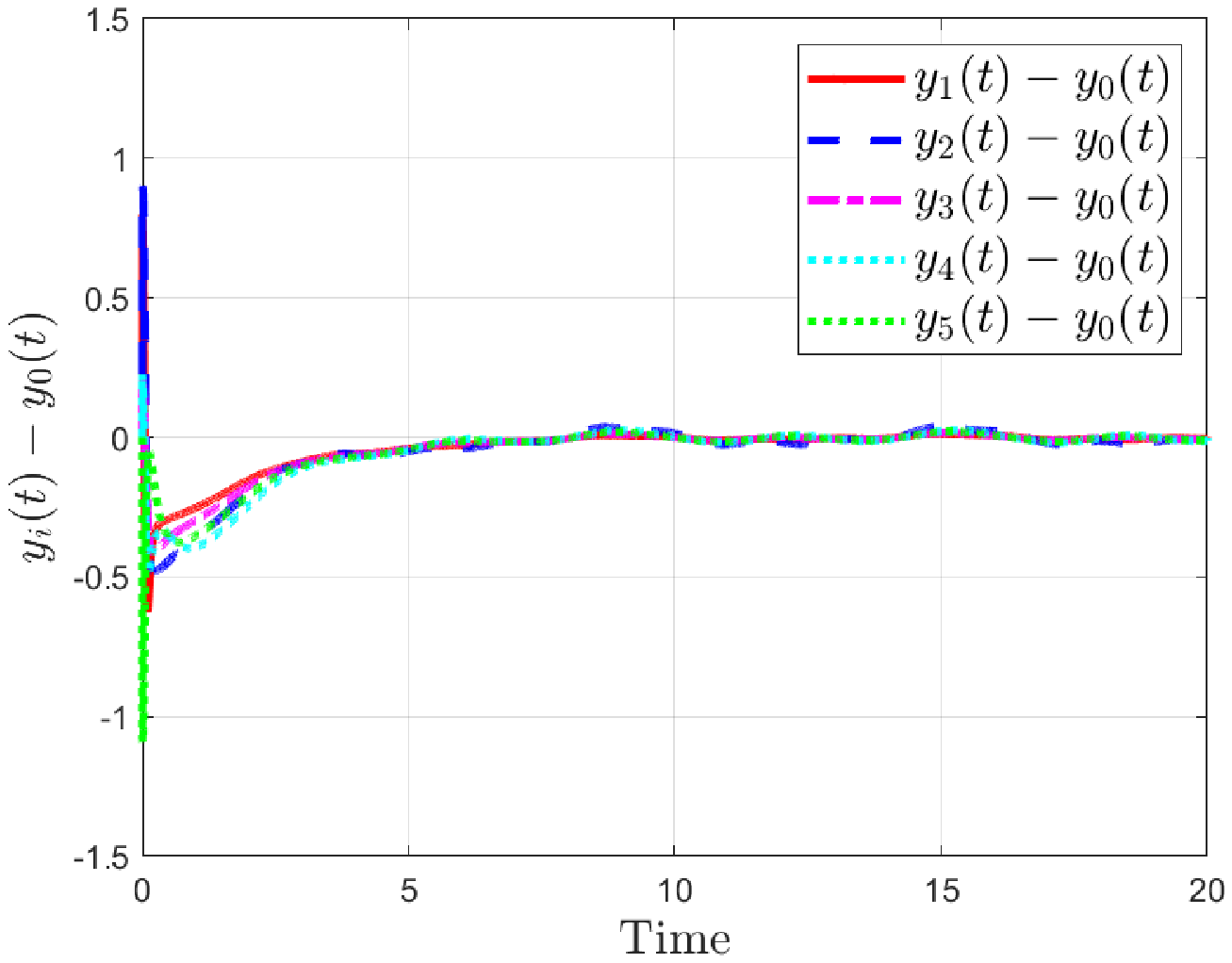}}
	\subfigure[The trajectories of the second state of followers]
	{\includegraphics[width=4.3cm,height=4cm]{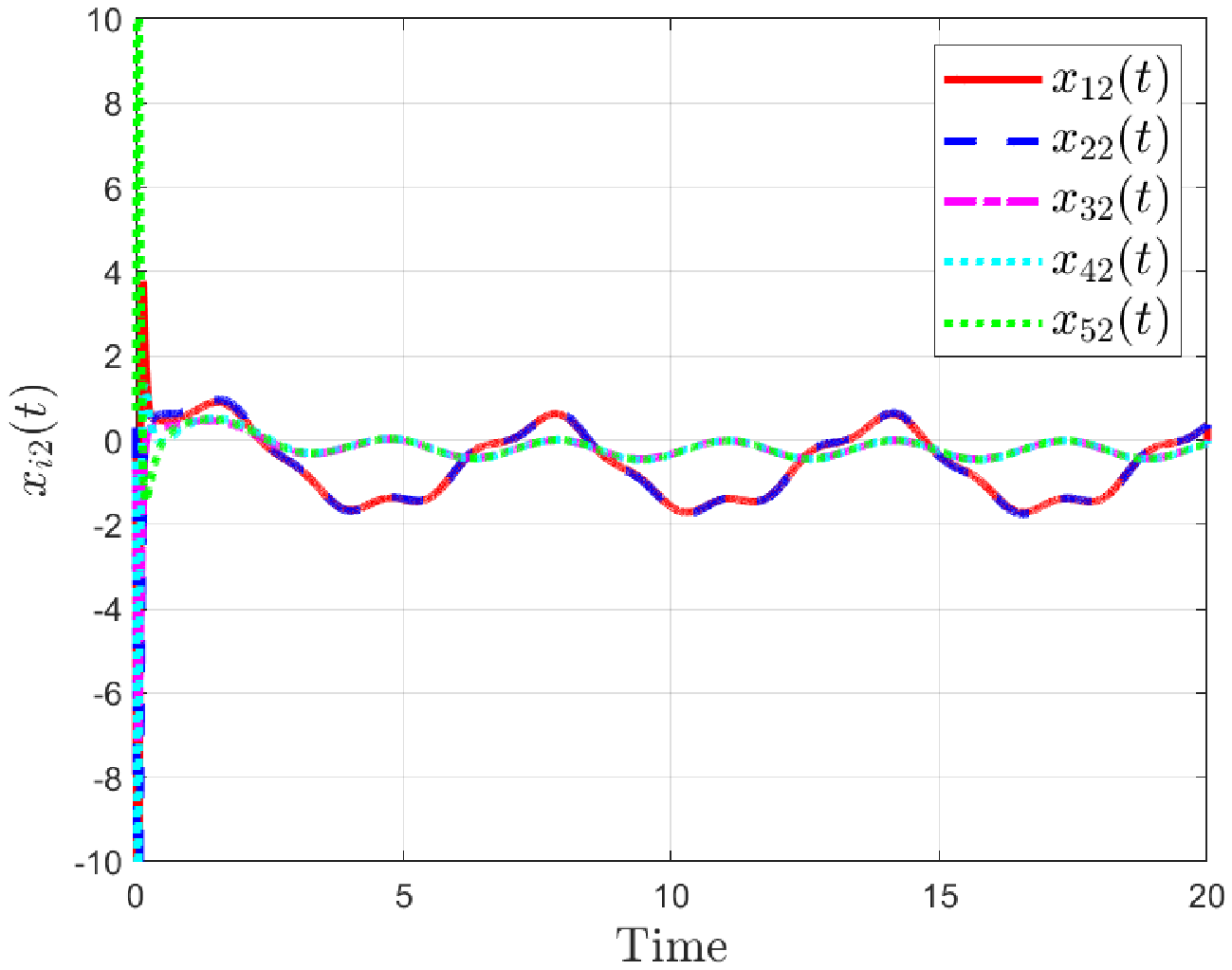}}
	\subfigure[The trajectories of estimation errors of actual total disturbances]
	{\includegraphics[width=4.3cm,height=4cm]{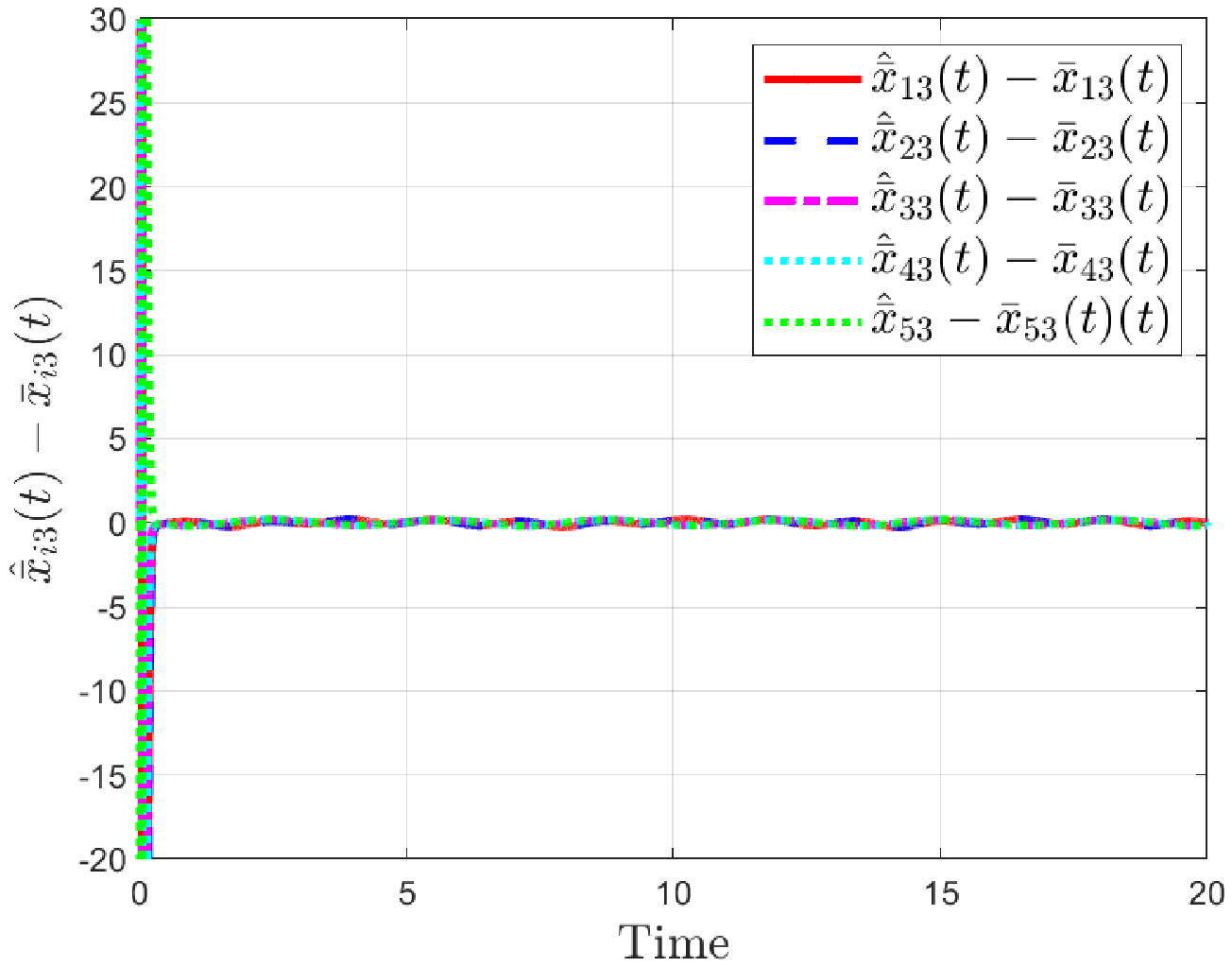}}
	\caption{The effects of output consensus, state bounedness, and estimation of actual total disturbances.}\label{trackingeffect3}
\end{figure}

\begin{figure}[!htb]\centering
	\subfigure[The trajectories of tracking errors of outputs]
	{\includegraphics[width=4.3cm,height=4cm]{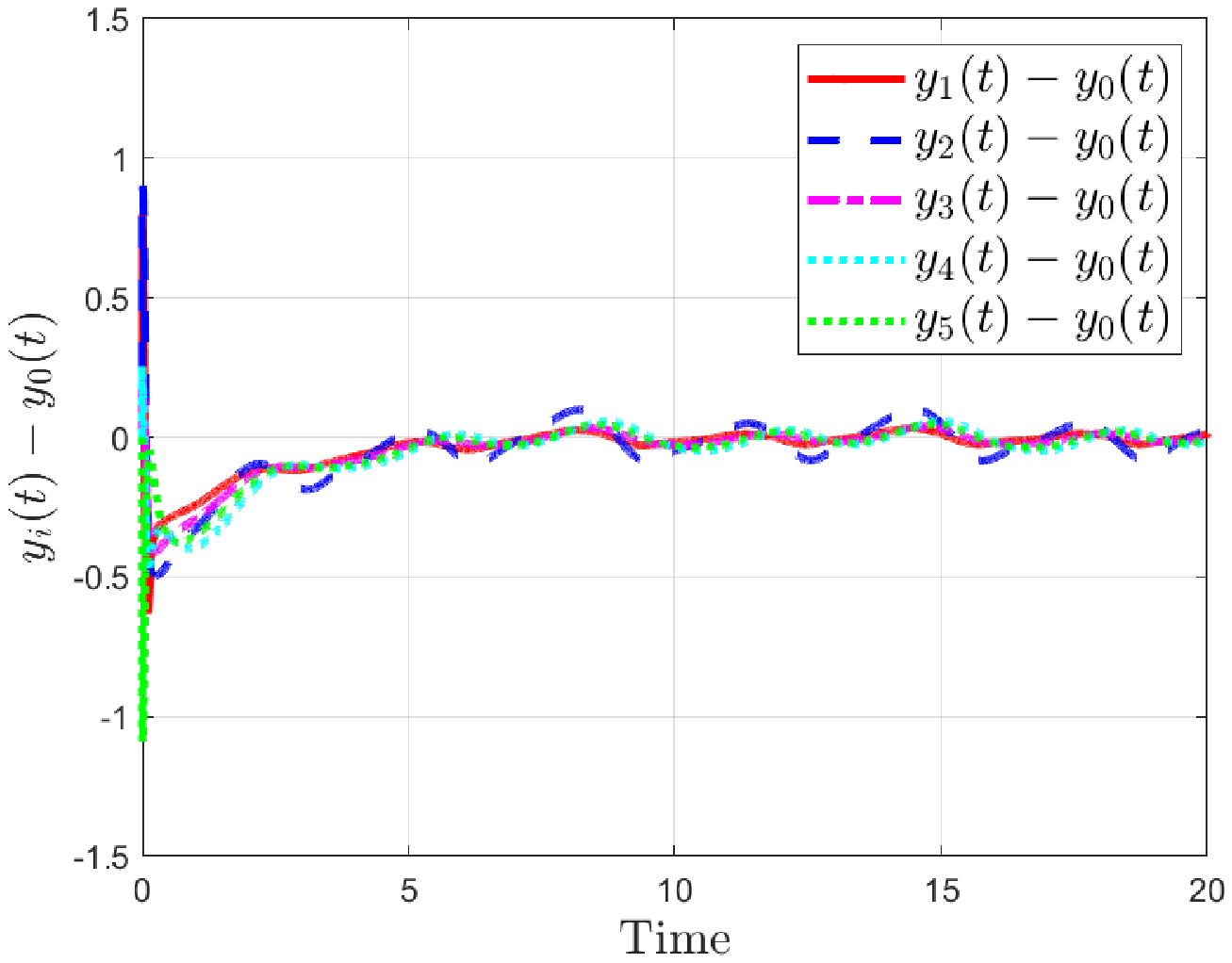}}
	\subfigure[The trajectories of the second state of followers]
	{\includegraphics[width=4.3cm,height=4cm]{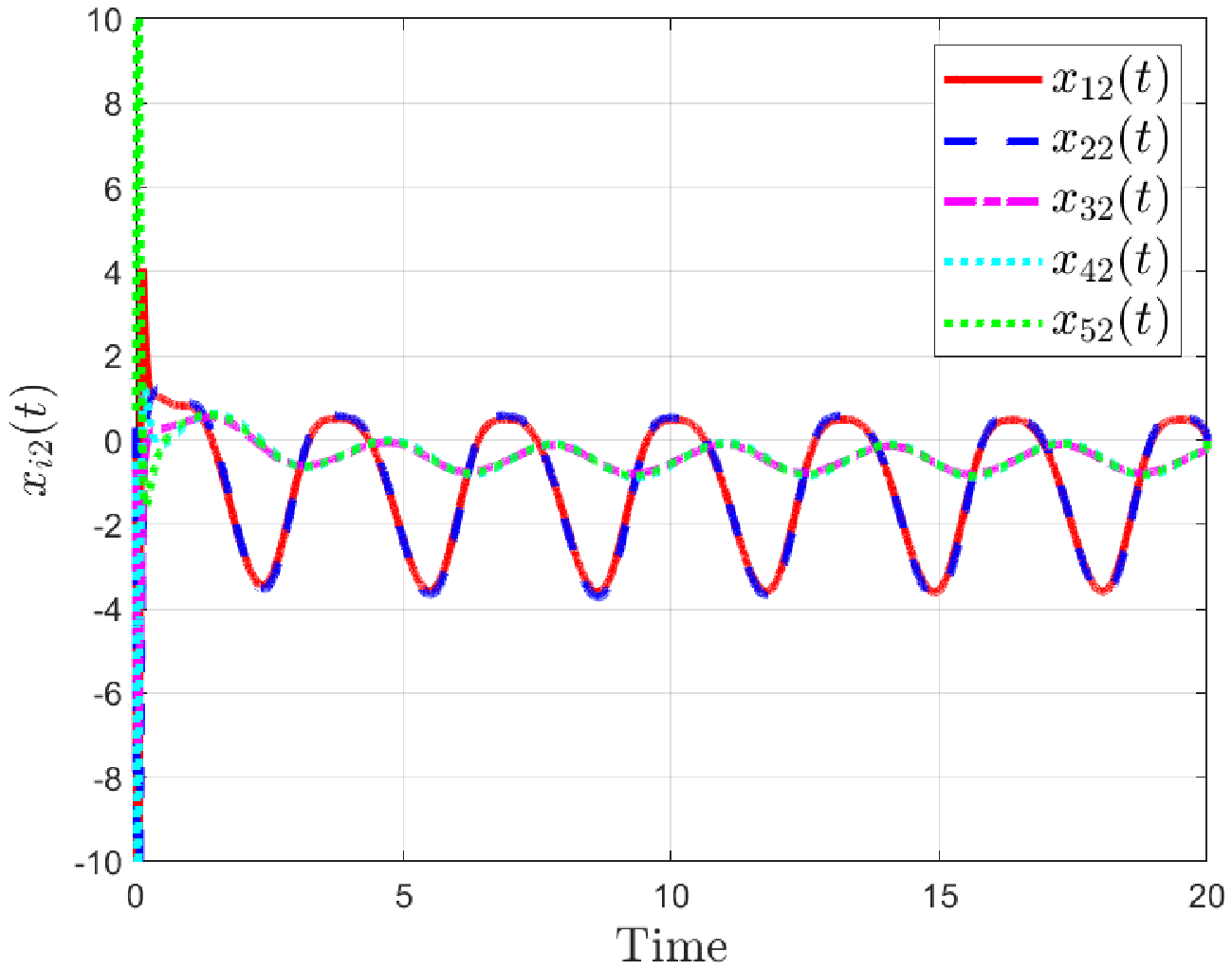}}
	\subfigure[The trajectories of estimation errors of actual total disturbances]
	{\includegraphics[width=4.3cm,height=4cm]{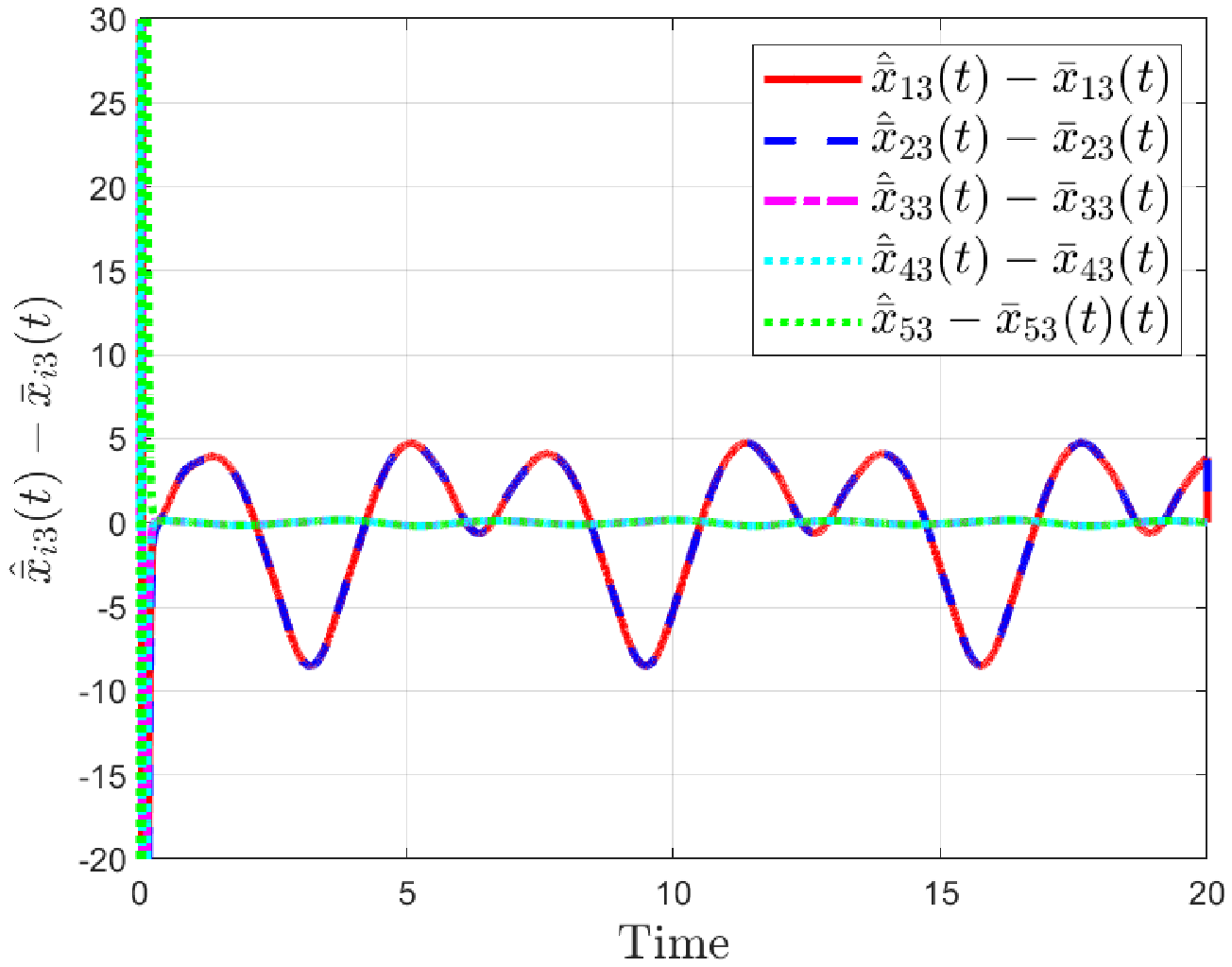}}
	\caption{The effects of output consensus, state bounedness, and estimation of actual total disturbances.}\label{trackingeffect4}
\end{figure}

\section{Concluding remarks}\label{se5}
This paper addresses the practical output consensus and disturbance rejection for
a class of heterogeneous high-order leader-follower MASs with mismatched disturbances, mismatched uncertainties,
and measurement noises in large scale. The network topology is directed and
containing a directed spanning tree. A set of ESOs
are designed using only the output measurement of each agent are
designed to estimate the actual total disturbance of each agent in real time,
and then the ADRC consensus protocols based on ESOs are designed, guaranteing
 that the outputs of all followers can track practically the output of the leader
and all the states of the leader-follower MASs are bounded. Finally, the availability
of the ADRC consensus protocols and the rationality of the theoretical result
are confirmed by some numerical simulations.

%
%

\ifCLASSOPTIONcaptionsoff
  \newpage
\fi




\begin{thebibliography}{1}

\bibitem{Cao2015}Cao, W., Zhang, J., \& Ren, W. (2015). Leader-follower consensus of linear multi-agent systems with unknown external disturbances. {\it Systems \& Control Letters}, 82, 64-70.


\bibitem{yu2010}Yu, W., Chen, G., \& Cao, M. (2010). Some necessary and sufficient conditions for second-order consensus in multi-agent dynamical systems. {\it Automatica}, 46(6), 1089-1095.


\bibitem{yu2011} Yu, W., Chen, G., \& Cao, M. (2011). Consensus in directed networks of agents with nonlinear dynamics. {\it IEEE Transactions on Automatic Control}, 56(6), 1436-1441.


\bibitem{fan2014}Fan, M. C., Chen, Z., \& Zhang, H. T. (2014). Semi-global consensus of nonlinear second-order multi-agent systems with measurement output feedback. {\it IEEE Transactions on Automatic Control}, 59(8), 2222-2227.


\bibitem{limulti-agent} Li, T., Wu, F., \& Zhang, J. F. (2014). Multi-agent consensus with relative-state-dependent measurement noises. {\it IEEE Transactions on Automatic Control}, 59(9), 2463-2468.

\bibitem{2nonlinearmulti-agent} Shen, Q., \& Shi, P. (2015). Distributed command filtered backstepping consensus tracking control of nonlinear multiple-agent systems in strict-feedback form. {\it Automatica}, 53, 120-124.


\bibitem{3stochasticmulti-agent} Li, M., \& Deng, F. (2019). Necessary and sufficient conditions for consensus of continuous-time multiagent systems with markovian switching topologies and communication noises. {\it IEEE Transactions on Cybernetics}, 50(7), 3264-3270.
    
\bibitem{zong2019}Zong, X., Li, T., \& Zhang, J. F. (2019). Consensus conditions of continuous-time multi-agent systems with time-delays and measurement noises. {\it Automatica}, 99, 412-419.


\bibitem{dengcwuzg2020} Deng, C., Che, W. W., \& Wu, Z. G. (2020). A dynamic periodic event-triggered approach to consensus of heterogeneous linear multiagent systems with time-varying communication delays. {\it IEEE Transactions on Cybernetics}, 51(4), 1812-1821.
    
\bibitem{In6}Dong, X., Hua, Y., Zhou, Y., Ren, Z., \& Zhong, Y. (2018). Theory and experiment on formation-containment control of multiple multirotor unmanned aerial vehicle systems. {\it IEEE Transactions on Automation Science and Engineering}, 16(1), 229-240.

\bibitem{In7}Das, A., Gervet, T., Romoff, J., Batra, D., Parikh, D., Rabbat, M., \& Pineau, J. (2019, May). Tarmac: Targeted multi-agent communication. {\it In International Conference on Machine Learning}, 1538-1546. PMLR.


\bibitem{In8}Liang, L., Ye, H., \& Li, G. Y. (2019). Spectrum sharing in vehicular networks based on multi-agent reinforcement learning. {\it IEEE Journal on Selected Areas in Communications}, 37(10), 2282-2292.

\bibitem{powerapplication}McArthur, S. D., Davidson, E. M., Catterson, V. M., Dimeas, A. L., Hatziargyriou, N. D., Ponci, F., \& Funabashi, T. (2007). Multi-agent systems for power engineering applications-Part II: Technologies, standards, and tools for building multi-agent systems. {\it IEEE Transactions on Power Systems}, 22(4), 1753-1759.

\bibitem{liTAC2014}Li, Z., Wen, G., Duan, Z., \& Ren, W. (2014). Designing fully distributed consensus protocols for linear multi-agent systems with directed graphs. {\it IEEE Transactions on Automatic Control}, 60(4), 1152-1157.

\bibitem{adptiveconl2017}Liu, W., \& Huang, J. (2017). Adaptive leader-following consensus for a class of higher-order nonlinear multi-agent systems with directed switching networks. {\it Automatica}, 79, 84-92.

\bibitem{In13}Mondal, S., Su, R., \& Xie, L. (2017). Heterogeneous consensus of higher-order multi-agent systems with mismatched uncertainties using sliding mode control. {\it International Journal of Robust and Nonlinear Control}, 27(13), 2303-2320.

\bibitem{slidecontrol2}Guo, X. G., Tan, D. C., Ahn, C. K., \& Wang, J. L. (2022). Fully distributed adaptive fault-tolerant sliding-mode control for nonlinear leader-following multiagent systems with ANASs and IQCs. {\it IEEE Transactions on Cybernetics}, 52(5), 2763-2774.



\bibitem{zhangh2014}Zhang, H., Zhang, J., Yang, G. H., \& Luo, Y. (2014). Leader-based optimal coordination control for the consensus problem of multiagent differential games via fuzzy adaptive dynamic programming. {\it IEEE Transactions on Fuzzy Systems}, 23(1), 152-163.


\bibitem{outputregulation}Liu, W., \& Huang, J. (2017). Cooperative adaptive output regulation for second-order nonlinear multiagent systems with jointly connected switching networks. {\it IEEE Transactions on Neural Networks and Learning Systems}, 29(3), 695-705.



\bibitem{ADRC}Han, J. (2009). From PID to active disturbance rejection control. {\it IEEE transactions on Industrial Electronics}, 56(3), 900-906.


\bibitem{DOBC}Li, S., Yang, J., Chen, W. H., \& Chen, X. (2014). {\it Disturbance Observer-based Control: Methods and Applications.} CRC press, Boca Raton.

\bibitem{ESOmulti-agentsystems} Wang, C., Zuo, Z., Qi, Z., \& Ding, Z. (2018). Predictor-based extended-state-observer design for consensus of MASs with delays and disturbances. {\it IEEE Transactions on Cybernetics}, 49(4), 1259-1269.


\bibitem{zou2022}Zou, A. M., Liu, Y., Hou, Z. G., \& Hu, Z. (2022). Practical predefined-time output-feedback consensus tracking control for multiagent systems. {\it IEEE Transactions on Cybernetics},
early access,  October 6, 2022, doi: 10.1109/TCYB.2022.3207325.


\bibitem{limengling2022}Li, M., Wu, Z. H., Deng, F., \& Guo, B. Z. (2022). Active Disturbance Rejection Control to Consensus of Second-Order Stochastic Multi-Agent Systems. {\it IEEE Transactions on Control of Network Systems},
early access, October 11, 2022, doi: 10.1109/TCNS.2022.3213710.

\bibitem{hua2018}Hua, Y., Dong, X., Han, L., Li, Q., \& Ren, Z. (2018). Finite-time time-varying formation tracking for high-order multiagent systems with mismatched disturbances. {\it IEEE Transactions on Systems, Man, and Cybernetics: Systems}, 50(10), 3795-3803.

\bibitem{mismatched1}Yang, J., Su, J., Li, S., \& Yu, X. (2013). High-order mismatched disturbance compensation for motion control systems via a continuous dynamic sliding-mode approach. {\it IEEE Transactions on Industrial Informatics}, 10(1), 604-614.

\bibitem{mismatched2} Chen, W. H. (2003). Nonlinear disturbance observer-enhanced dynamic inversion control of missiles. {\it Journal of Guidance, Control, and Dynamics}, 26(1), 161-166.

\bibitem{xue2014}Xue, W., \& Huang, Y. (2014). On performance analysis of ADRC for a class of MIMO lower-triangular nonlinear uncertain systems. {\it ISA Transactions}, 53(4), 955-962.

\bibitem{guo2017} Guo, B. Z., \& Wu, Z. H. (2017). Output tracking for a class of nonlinear systems with mismatched uncertainties by active disturbance rejection control. {\it Systems \& Control Letters}, 100, 21-31.


\bibitem{zhao2017}Zhao, Z. L., \& Guo, B. Z. (2017). A novel extended state observer for output tracking of MIMO systems with mismatched uncertainty. {\it IEEE Transactions on Automatic Control}, 63(1), 211-218.

\bibitem{chens2020}Chen, S., \& Chen, Z. (2020). On active disturbance rejection control for a class of uncertain systems with measurement uncertainty. {\it IEEE Transactions on Industrial Electronics}, 68(2), 1475-1485.

\bibitem{wu2021}Wu, Z. H., Deng, F., Guo, B. Z., Wu, C., \& Xiang, Q. (2021). Backstepping active disturbance rejection control for lower triangular nonlinear systems with mismatched stochastic disturbances. {\it IEEE Transactions on Systems, Man, and Cybernetics: Systems}, 52(4), 2688-2702.
    

\bibitem{misMAS1}Wang, X., Li, S., \& Lam, J. (2016). Distributed active anti-disturbance output consensus algorithms for higher-order multi-agent systems with mismatched disturbances. {\it Automatica}, 74, 30-37.


\bibitem{misMAS2}Wang, X., Li, S., \& Chen, M. Z. (2017). Composite backstepping consensus algorithms of leader-follower higher-order nonlinear multiagent systems subject to mismatched disturbances. {\it IEEE Transactions on Cybernetics}, 48(6), 1935-1946.

\bibitem{Yandyue} Yang, Y., Si, X., Yue, D., \& Tan, J. (2020). Time-varying formation tracking of uncertain nonaffine nonlinear multiagent systems with communication delays. {\it IEEE Transactions on Industrial Electronics}, 68(3), 2501-2509.

\bibitem{ran2021}Ran, M., \& Xie, L. (2021). Practical output consensus of nonlinear heterogeneous multi-agent systems with limited data rate. {\it Automatica}, 129, 109624.

\bibitem{exactobserval}Cheng, D., Hu, X., \& Shen, T. (2010). {\it Analysis and Design of Nonlinear Control Systems.} Science Press.

\end{thebibliography}
%

%






\end{document}